\documentclass[greek,english,11pt,a4paper,openany]{article}
\usepackage{amssymb,latexsym, amsfonts, amsmath, amsthm}
\usepackage{graphicx}
\usepackage{pgfpages}
\usepackage{setspace}
\usepackage{enumerate}
\usepackage{wrapfig}
\usepackage{lscape}
\usepackage[paperwidth=210mm, paperheight=297mm, top=1in, bottom=1in, left=1in, right=1in, asymmetric, marginparwidth=0.9in]{geometry}
\usepackage{rotating}
\usepackage{epstopdf}
\usepackage[backgroundcolor=green!40, linecolor=green!40]{todonotes}
\usepackage{tikz}
\usetikzlibrary{arrows,calc,matrix,fit}
\theoremstyle{definition}
\newtheorem{thm}[equation]{Theorem}

\newtheorem*{theoremA}{Theorem A}
\newtheorem*{theoremB}{Theorem B}
\newtheorem*{theoremC}{Theorem C}
\newtheorem*{theoremD}{Theorem D}
\newtheorem*{ack}{Acknowledgements}

\newtheorem{lemma}[equation]{Lemma}

\newtheorem{corollary}[equation]{Corollary}

\newtheorem{remark}[equation]{Remark}

\numberwithin{equation}{section} 

\def\End{\operatorname{End}}
\def\Aut{\operatorname{Aut}}
\def\Ind{\operatorname{Ind}}
\def\U{\operatorname{U}}
\def\Hom{\operatorname{Hom}}

\def\Lie{\operatorname{Lie}}
\def\dim{\operatorname{dim}}
\def\infl{\operatorname{infl}}
\def\Id{\operatorname{Id}}
\def\Ql{\overline{\mathbb{Q}}_{\ell}}

\def\Fl{\overline{\mathbb{F}}_{\ell}}
\def\Gal{\operatorname{Gal}}
\def\GL{\operatorname{GL}}

\def\ddag{{\,\dag}}
\def\Tr{\operatorname{Tr}}

\def\Res{\operatorname{Res}}
\def\Ind{\operatorname{Ind}}
\def\GE{G_E}

\def\ind{\operatorname{ind}}
\def\I{\operatorname{I}}
\def\R{\operatorname{R}}

\def\CC{\mathbb{C}}

\def\bbM{\mathbb{M}}

\def\GG{\mathbb{G}}
\def\ZZ{\mathbb{Z}} 

\def\J{{\rm J}}
\def\K{{\rm K}}

\def\Cc{\mathcal{C}}

\def\Hh{\mathcal{H}}

\def\Ll{\mathcal{L}}
\def\Mm{\mathcal{M}}

\def\Vv{\mathcal{V}}
\def\Ww{\mathcal{W}}

\def\AA{\mathfrak{A}}
\def\BB{\mathfrak{B}}
\def\HH{\mathfrak{H}}
\def\JJ{\mathfrak{J}}

\def\MM{\mathfrak{M}}

\def\QQ{\mathfrak{Q}}
\def\RR{\mathfrak{R}}

\def\a{\alpha} 
\def\b{\beta}

\def\e{\varepsilon}
\def\g{\gamma}

\def\k{\kappa}
\def\l{\lambda}

\def\o{\mathfrak{o}}
\def\p{\mathfrak{p}}
\def\s{\sigma}
\def\t{\tau}
\def\th{\theta}

\def\Ga{\Gamma}
\def\La{\Lambda}

\def\Y{\Upsilon}

\def\Si{\Sigma}
\def\om{\omega}

\def\>{\geqslant}
\def\<{\leqslant}

\def\nn{\mathfrak{n}}

\def\ie{\emph{i.e.}}

\def\cf{\emph{cf.}}

\def\tG{\widetilde{G}}
\def\tH{\widetilde{H}}
\def\tJ{\widetilde{J}}
\def\tK{\widetilde{K}}
\def\tP{\widetilde{P}}
\def\tQ{\widetilde{Q}}
\def\tU{\widetilde{U}}
\def\tM{\widetilde{M}}
\def\tL{\widetilde{L}}
\def\tGE{\tG_E}
\def\tth{\widetilde{\th}}
\def\teta{\widetilde{\eta}}
\def\tl{\widetilde{\l}}
\def\trho{\widetilde{\rho}}
\def\tk{\widetilde{\k}}
\def\tchi{\widetilde{\chi}}
\def\ttau{\widetilde{\t}}

\def\ov#1{\overline{#1}}
\def\({\left(}
\def\){\right)}

\def\presuper#1#2%
  {\mathop{}%
   \mathopen{\vphantom{#2}}^{#1}%
   \kern-\scriptspace%
   #2}
\makeatletter
\tikzstyle{notestyleraw}=[
  draw=\@todonotes@currentbordercolor,
  fill=\@todonotes@currentbackgroundcolor,
  line width=0.5pt,
  text width=\@todonotes@textwidth+9ex,
  inner sep=0.8ex,
  rounded corners=4pt
]
\makeatother
\begin{document}

\title{Cuspidal~$\ell$-modular representations of~$p$-adic classical groups}
\author{Robert Kurinczuk\footnote{Robert Kurinczuk, Heilbronn Institute for Mathematical Research, Department of Mathematics, University of Bristol, BS8 1TW, United Kingdom. Email: robkurinczuk@gmail.com},  Shaun Stevens\footnote{Shaun Stevens, School of Mathematics, University of East Anglia, Norwich, NR4 7TJ, United~Kingdom.~Email: shaun.stevens@uea.ac.uk}}
\maketitle

\begin{abstract}
\noindent For a classical group over a non-archimedean local field of odd residual characteristic~$p$, we construct all cuspidal representations over an arbitrary algebraically closed field of characteristic different from~$p$, as representations induced from a cuspidal type. We also give a fundamental step towards the classification of cuspidal representations, identifying when certain cuspidal types induce to equivalent representations; this result is new even in the case of complex representations. Finally, we prove that the representations induced from more general types are quasi-projective, a crucial tool for extending the results here to arbitrary irreducible representations.
\end{abstract}
\section{Introduction}
In recent years, congruences between automorphic representations have assumed a central importance in number theory.  This has led to the desire to understand representations of reductive~$p$-adic groups on vector spaces over fields of positive characteristic~$\ell$. There are vast differences between the cases~$\ell=p$ and~$\ell\neq p$, with the latter sharing many similarities with the theory of complex representations, including the existence of a Haar measure. However, there are also many important and interesting differences between the~$\ell\neq p$ theory and the theory for complex representations, including the presence of compact open subgroups of measure zero, the non-semisimplicity of smooth representations of compact open subgroups, and that cuspidal representations can and do appear as subquotients of parabolically induced representations (in fact, all of these phenomena are related).  In this article we focus on the~$\ell\neq p$ case, and work with an arbitrary algebraically closed field of characteristic~$\ell$ or zero.

The theory of (smooth) representations of a general reductive~$p$-adic group over such fields was developed by Vign\'eras in~\cite{Vig96}.  However, many subsequent articles and fundamental results (for example, the unicity of supercuspidal support) focus just on the general linear group.  One of the main reasons that this group has been more accessible for a modular theory, is that the Bushnell--Kutzko classification of irreducible complex representations via types extends in a natural way to~$\ell$-modular representations, which is the subject of the final chapter of [ibid.].  This classification, in favourable circumstances, allows one to reduce a problem to an analogous question in associated finite groups where hopefully it is either tractable to the pursuer, or already known.  Recently, this approach has been adopted for other groups: S\'echerre and M\'inguez in~\cite{VA2} for inner forms of~$\GL_n$; and the first author in~\cite{Kurinczuk} for unramified~$\U(2,1)$.  In this article, we pursue this approach for~$p$-adic classical groups~$G$ over locally compact non-archimedean local fields with odd residual characteristic.

Of particular importance in this approach is the construction of all irreducible cuspidal~complex representations of general linear groups as compactly induced representations.   We accomplish this for $\ell$-modular representations in our main results:

\begin{theoremA}[Theorems \ref{cusptype},~\ref{exhaustion}]
There is an explicit list of \emph{cuspidal types}, consisting of certain pairs $(J,\lambda)$, with $J$ a compact open subgroup of $G$ and $\lambda$ an irreducible $R$-representation of $J$ such that 
\begin{enumerate}
\item\label{part1theoremA} the compactly induced representation $\ind_J^G\lambda$ is irreducible and cuspidal;
\item every irreducible cuspidal representation arises as in \ref{part1theoremA}, for some cuspidal type~$(J,\lambda)$.
\end{enumerate}
\end{theoremA}

See below for a more precise definition of cuspidal type.  For complex representations this is the main result of \cite{St08}.  But here we do more, giving an initial refinement of this exhaustive list of cuspidal types.  Part of the data used to define a cuspidal type is a family of \emph{skew semisimple characters}.  In the case where two cuspidal types are defined relative to the same family (see below for a more precisely-worded condition), we obtain the following \emph{intertwining implies conjugacy} result:

\begin{theoremB}[Theorem \ref{cuspcomp}]
Let~$(J_1,\lambda_1)$,~$(J_2,\lambda_2)$ be cuspidal types defined relative to the same family of skew semisimple characters.  Then $\ind_{J_1}^G\lambda_1\simeq \ind_{J_2}^G\lambda_2$ if and only if there exists $g\in G$ such that $J_1^g=J_2$ and $\lambda_1^g\simeq \lambda_2$.
\end{theoremB}

Note that $\lambda_1^g$ here denotes the representation of $J_1^g=g^{-1}J_1 g$ given by $\lambda_1^g(j)=\lambda_1(gj g^{-1})$, for $j\in J_1$.  In forthcoming joint work with Skodlerack, this theorem will be combined with work of the second author and Skodlerack to prove an intertwining implies conjugacy result without the condition on the skew semisimple characters.  We now give more details and explain our approach. 

Let~$G$ be a~$p$-adic classical group with~$p$ odd, that is (the points of) a unitary, symplectic or special orthogonal group defined over a locally compact non-archimedean local field $F$ of residual characteristic~$p$.  Let $\beta\in \Lie G$ be a semisimple element, and put $\GE=C_G(\beta)$ the $G$-centraliser of $\beta$.  Let $\Lambda$ be an $\o_F$-lattice sequence corresponding to a point in the Bruhat--Tits building of $\GE$.  From $\beta$ and $\La$ we get a set of \emph{self-dual semisimple characters} $\theta_{\Lambda}$ of a group $H^1_{\Lambda}$; and given another lattice sequence $\Upsilon$ as above, there is a canonical \emph{transfer map} giving a corresponding self-dual semisimple character $\theta_{\Y}$ of $H^1_{\Y}$.  Also write $J_{\Lambda}$ for the normaliser of $\theta_{\La}$ in the (non-connected) parahoric subgroup of $G$ corresponding to $\La$, and $J^1_{\La}$ for its pro-$p$ radical.  There is a unique irreducible representation $\eta_{\Lambda}$ of $J^1_{\La}$ which contains $\theta_{\La}$ on restriction.  Our first major diversion from the earlier results of the second author is:


\begin{theoremC}[{Theorems \ref{classthetaint} \& \ref{intetas}}]With notation as above.
\begin{enumerate}
\item The intertwining of $\theta_{\Lambda}$ with $\theta_{\Upsilon}$ is $J_{\Upsilon}\GE J_{\Lambda}$.
\item The intertwining spaces of $\eta_{\Lambda}$ with $\eta_{\Upsilon}$ are at most one dimensional; more precisely:
\[
\dim_{R}\Hom_{J_{\Lambda}^1\cap (J_{\Y}^1)^g}(\eta_\La,\eta_\Y^g)=
\begin{cases} 
1&\text{if }g\in J_\Y\GE J_\La;\\ 
0&\text{otherwise.}
\end{cases}
\]
\end{enumerate}
\end{theoremC}

This theorem is an asymmetric generalisation of~\cite[Propositions 3.27 \& 3.31]{St05} (\cf~also~\cite{MiSt}) which deals with the case~$\La=\Y$.  It appears possible, and indeed it is already hinted at in~\cite[1.5.12]{BK93}, that one could prove such an intertwining result by developing the theory \emph{ab initio}, with lattice sequences such as these rather than just a single lattice sequence.  However, our approach is more brief and elegant, utilising a construction for semisimple characters to relate the case of not necessarily conjugate lattice sequences to the case of conjugate lattice sequences in a larger group.  This construction is inspired by a similar one for simple strata, in work of the second author with Broussous and S\'echerre \cite{BSS}.  

%
%
%

The next step is to extend $\eta_{\La}$ to a suitable representation of $J_{\Lambda}$, called a \emph{$\beta$-extension}, which is accomplished in Section \ref{betasects}.  While we have to change the proofs of~\cite{St08} here, the changes are straightforward.  That the formation of covers, of~\cite{St08} and~\cite{MiSt}, is still valid in positive characteristic is proved in Sections \ref{skewcoverssect} and \ref{selfdualcovers}.  Let $\kappa_{\La}$ be a $\beta$-extension of $\eta_{\Lambda}$.  The quotient $J_{\La}/J_{\La}^1$ is a product of finite reductive groups and we write $J^{\circ}$ or the inverse image of the connected component.  Let $\tau$ be an irreducible representation of $J_{\La}/J_{\La}^1$ with cuspidal restriction to $J^{\circ}_{\La}/J^1_{\La}$, and put $\lambda=\kappa_{\La}\otimes\tau$ and $J=J_{\La}$. We call the pair $(J,\lambda)$ a \emph{type}; and if the centraliser~$\GE$ has compact centre, and the corresponding (connected) parahoric subgroup~$J^{\circ}_{\La}\cap\GE$ is maximal, we call the pair~$(J,\lambda)$ a \emph{cuspidal type}.

Finally, we are able to extend the main result of the second author in~\cite{St08} to~$\ell$-modular representations (see Theorem A).  Our approach to proving Theorem A is different to~\cite{St08} at the top level of the construction, relying on a \emph{reduction to level zero} argument (see Section \ref{Redlevzerosect}).  Thanks to our work in this paper on asymmetric intertwining of semisimple characters and Heisenberg representations, this new approach allows us to compare cuspidal representations in this exhaustive list whose semisimple characters are in the same family (i.e. are related by the transfer map), and make an initial refinement of the exhaustive list (see Theorem B).

We now mention further results we prove with future work in mind.  In the $\ell$-modular setting, compactly induced representations from types may not be projective.  This provides an obstruction to following Bushnell--Kutzko's approach via covers to the admissible dual, as the category of representations containing a type $(J,\lambda)$, will not in general be equivalent to the the category of right modules over the algebra $\End_G(\ind_J^G\lambda)$.  Following M\'inguez--S\'echerre we construct covers on pro-$p$ groups (Theorem \ref{Gcovers2}); these will have the advantage of providing such an equivalence of categories to the category of modules over an algebra as above. It may be that this algebra will prove unwieldy for classification purposes, but it can be related to a similar algebra in depth zero. For general linear groups, promising initial results in this direction have recently been obtained by Chinello in his thesis \cite{GC15}, while Dat has begun a detailed study of the depth zero subcategory in \cite{Dattame}.  Writing~$\lambda^\circ$ for an irreducible component of the restriction of~$\lambda$ to~$J^\circ$, we thus show:
\begin{theoremD}[Theorem 10.2] 
The representation $\ind_{J^{\circ}}^G \lambda^\circ$ is quasi-projective.  
\end{theoremD}
Thanks to work of Vign\'eras and Arabia \cite{vignerasselecta}, this implies that the irreducible quotients of $\ind_{J^{\circ}}^G \lambda^\circ$ are in bijection with the simple right modules of $\End_G(\ind_{J^{\circ}}^G \lambda^\circ)$, (see Section~\ref{background} for details).  As any irreducible representation of $G$ is a quotient of such an induced representation, this result is the starting point of an approach to classifying all irreducible $\ell$-modular representations of $G$.  

\begin{ack}
This work was supported by the Engineering and Physical Sciences Research Council (EP/H00534X/1) and the Heilbronn Institute for Mathematical Research.
\end{ack}

\section{Notation and background}\label{background}

Let~$F_0$ be a non-archimedean local field of odd residual
characteristic~$p$ and let~$F$ be either~$F_0$ or a quadratic
extension of~$F_0$.  Let~$\ov{\phantom{w}}$ denote the generator
of~$\Gal(F/F_0)$.  If~$E$ is a non-archimedean local field we denote
by~$\o_E$ the ring of integers of~$E$, by~$\p_E$
the unique maximal ideal of~$\o_E$, by~$k_E$ the residue
field and by~$q_E$ the cardinality of~$k_E$.  We
write~$\o_0=\o_{F_0}$, and similarly
abbreviate~$\p_0,k_0,q_0$. We fix a uniformizer~$\varpi_F$
of~$F$ such that~$\ov{\varpi_F}=-\varpi_F$ if~$F/F_0$ is
ramified and~$\ov{\varpi_F}=\varpi_F$ otherwise.  We fix a
character~$\psi_0$ of the additive group~$F_0$ with
conductor~$\p_0$ and let~$\psi_F=\psi_0\circ \Tr_{F/F_0}$.

Let~$V$ be an~$N$-dimensional~$F$-vector space equipped with a
non-degenerate~$\e$-hermitian form~$h:V\times V\to F$ with~$\e=\pm1$.
Let~$A=\End_F(V)$ and~$\tG=\Aut_F(V)$. 
The group~$G^+=\{g\in G: h(gv,gw)=h(v,w)\text{ for all }v,w\in V\}$ is
the~$F_0$-points of a unitary, symplectic or orthogonal algebraic
group~$\GG^+$ defined over~$F_0$. We let~$G$ denote
the~$F_0$-points of the connected component of~$\GG^+$ and
call~$G$ a \emph{classical group}.  Hence the special orthogonal group
is a classical group whereas the full orthogonal group is not. 

Let~$\ov{\phantom{w}}$ denote the adjoint (anti)-involution
induced on~$A$ by~$h$ and let~$A^{-}=\{a\in
A:a+\ov{a}=0\}\simeq\text{Lie}(G)$.  Let~$\s$ denote both
the involution on~$\tG$ defined by~$\s:g\mapsto
\ov{g}^{-1}$, for~$g\in \tG$, and its
derivative~$a\mapsto -\ov{a}$, for~$a\in A$.  Let~$\Si$ be
the cyclic group of order two generated by~$\s$.
Then~$G^+=\tG^{\Si}$ and~$A^-=A^{\Si}$.  We
have~$A=A^{-}\oplus A^{+}$ where~$A^+=\{a\in A:
a-\ov{a}=0\}$. We let~$\psi_A=\psi_F\circ \Tr_{A/F}$.  If~$S$ is
a subset of~$A$, we let~$S^*=\{x\in A:\psi_A(xS)=1\}$.

We let~$R$ denote an algebraically closed field of
characteristic~$\ell$ different from~$p$, allowing the case~$\ell=0$. For any locally compact topological group~$H$, we denote by~$\RR_R(H)$ the category of smooth~$R$-representations of~$H$.

\subsection{Representations and Hecke algebras}

For general results on representations of reductive~$p$-adic groups
over an algebraically closed field of characteristic different from~$p$, we
refer to Vign\'eras's book~\cite{Vig96}.

Let~$G$ be a reductive~$p$-adic group.  Let~$K, K_1,K_2$ be compact
open subgroups of~$G$,~$(\t,\Ww)$ be a smooth~$R$-representation of~$K$, and~$(\t_i,\Ww_i)$ be
smooth~$R$-representations of~$K_i$, for~$i=1,2$. For~$g\in G$,
the~$g$-\emph{intertwining space} of~$\t_1$ with~$\t_2$ is defined to
be the set
\[
I_g(\t_1,\t_2)=\Hom_{K_1\cap K_2^g}(\t_1,\t_2^g),
\]
and the \emph{intertwining} of~$\t_1$ with~$\t_2$ in~$G$ is 
\[
I_G(\t_1,\t_2)=\{g\in G: I_g(\t_1,\t_2)\neq0\},
\]
where~$K_2^g=g^{-1}K_2 g$ and~$\t_2^g(x)=\t_2(gxg^{-1})$
for~$x\in K_2^g$.  For an~$R$-representation~$(\pi,\Vv)$ of a
locally profinite group we denote by~$(\pi^{\vee},\Vv^{\vee})$
its contragredient representation.

\begin{remark}
The motivation for this definition is provided by the following decomposition
\[\Hom_G(\ind_{K_1}^G(\tau_1),\ind_{K_2}^G(\tau_2))\simeq\bigoplus_{K_2\backslash I_G(\tau_1,\tau_2)/K_1} I_g(\tau_1,\tau_2),\]
 by reciprocity and Mackey theory.  Note that, if~$K=K_1=K_2$,~$\tau=\tau_1=\tau_2$ and~$g\in G$, for complex representations or if~$K$ is pro-$p$, the spaces~$I_g(\tau)=\Hom_{K\cap K^g}(\t,\t^g)$ and~$\Hom_{K\cap \presuper{g}K}(\t,\presuper{g}\t)$ are the same, as representations of~$K\cap \presuper{g}K=K\cap K^g$ are semisimple, so in previous works one sees intertwining defined in either way. 
\end{remark}

Suppose that~$K_1$ and~$K_2$ are normal open subgroups of~$K$.
Let~$\Hh(G,\t_1,\t_2)$ be the~$R$-vector space of
compactly supported functions~$f:G\to
\Hom_R(\Ww_1,\Ww_2)$ which transform on the left
by~$\t_2$ and on the right
by~$\t_1$. Let~$\Hh(G,\t)=\Hh(G,\t,\t)$
denote the~$R$-algebra consisting of compactly supported
functions~$f:G\to \End_R(\Ww)$ which transform on the
left and the right by~$\t$ together with the convolution product 
\[
f_1\star f_2 (h)=\sum_{g\in G/K}f_1(g)f_2(g^{-1}h),
\]
for~$f_1,f_2\in\Hh(G,\t)$. This algebra has a unit element
if the index of every open subgroup in~$K$ is invertible
in~$R$ (\ie~the pro-order of~$K$ is invertible in~$R$).
The~$K$-invariant bilinear
pairing~$\langle\phantom{w},\phantom{w}\rangle$
on~$\Ww\times\Ww^{\vee}$ induces an
anti-isomorphism~$\Hh(G,\t)\to\Hh(G,\t^{\vee})$
by~$f\mapsto f^{\vee}$ with~$f^{\vee}$ defined by~$\langle
w,f^{\vee}(g^{-1})\check{w}\rangle=\langle f(g)w,\check{w}\rangle$
for all~$w\in\Ww$,~$\check{w}\in\Ww^{\vee}$.  Under
convolution~$\Hh(G,\t_1,\t_2)$ has
an~$(\Hh(G,\t_1),\Hh(G,\t_2))$-bimodule
structure.  If~$g\in G$, we let~$\Hh(G,\t_1,\t_2)_g$
denote the subspace of all functions with support~$K_1gK_2$.
 
Under composition,~$\End_G(\ind_K^G\t)$ has an~$R$-algebra
structure and~$\Hom_G(\ind_{K_1}^G\t_1,\ind_{K_2}^G\t_2)$ is 
an~$(\End_G(\ind_{K_1}^G\t_1),\End_G(\ind_{K_2}^G\t_2))$-bimodule.
The proof of the following Lemma follows from the proofs contained in
{\cite[\S8.5,~8.6,~\&~8.10]{Vig96}}.

\begin{lemma}\label{heckelemma} 
\begin{enumerate}
\item\label{heckelemma:1} We have an isomorphism of algebras
\[
\Hh(G,\t)\simeq \End_G(\ind_K^G\t).
\]
\item For~$i=1,2$, we identify~$\Hh(G,\t_i)$
with~$\End_G(\ind_{K_i}^G\t_i)$ by~\ref{heckelemma:1}.  We have
an isomorphism of~$(\Hh(G,\t_1),\Hh(G,\t_2))$-bimodules 
\[
\Hh(G,\t_1,\t_2)\simeq \Hom_G(\ind_{K_1}^G\t_1,\ind_{K_2}^G\t_2).
\]
\item\label{heckelemma:3} For~$i=1,2$, let~$H_i$ be compact open
subgroups of~$G$ containing~$K_i$.  We have an isomorphism
of~$(\Hh(G,\t_1),\Hh(G,\t_2))$-bimodules 
\[
 \Hh(G,\ind_{K_1}^{H_1}\t_1,\ind_{K_2}^{H_2}\t_2)\simeq\Hh(G,\t_1,\t_2),
\]
which restricts to give isomorphisms of vector spaces, for~$g\in G$, 
\[
\Hh(G,\ind_{K_1}^{H_1}\t_1,\ind_{K_2}^{H_2}\t_2)_g\simeq 
\coprod_{\substack{h\in H_1\backslash G/H_2\\K_1hK_2=K_1gK_2}}
\Hh(G,\t_1,\t_2)_h.
\]
\end{enumerate}
\end{lemma}

\subsection{Lattice sequences and parahoric subgroups}

An \emph{$\o_F$-lattice sequence} in~$V$ is a function 
\[
\La:\ZZ\to\{\o_F\text{-lattices in }V\}
\]
which is \emph{decreasing}, that is~$\La(n+1)\subseteq \La(n)$,
for all~$n\in\ZZ$, and \emph{periodic}, that is, there
exists a positive integer~$e(\La)$ such
that~$\La(n+e(\La))=\varpi_F\La(n)$, for all~$n\in\ZZ$.

The~$\e$-hermitian form~$h$ defines a duality on the set
of~$\o_F$-lattices; given an~$\o_F$-lattice~$L$ we
let~$L^\sharp=\{v\in V:h(v,L)\subseteq \p_F\}$.  An~$\o_F$-lattice
sequence~$\La$ is called \emph{self-dual} if~$\La(k)^\sharp=\La(1-k)$,
for all~$k\in\ZZ$.

An~$\o_F$-lattice sequence~$\La$ induces a decreasing filtration
on~$A$ by~$\o_F$-lattices~$\AA_n(\La)$ in~$A$ where
\[
\AA_n(\La)=\{x\in A:x\La(m)=\La(m+n),m\in\ZZ\},~\text{for }n\in\ZZ.
\]
This filtration induces a \emph{valuation} on~$A$ defined by
\[
\nu_{\La}(x)=\begin{cases}
\sup\{n\in\ZZ:x\in\AA_n(\La)\}&\text{if }x\in A\backslash\{0\};\\
\infty&\text{if }x=0.
\end{cases}
\]
If~$\La$ is self-dual, it induces a decreasing filtration on~$A^{-}$
by~$\o_F$-lattices~$\AA^{-}_n(\La)$ in~$A^{-}$ where
\[
\AA^{-}_n(\La)=\AA_n(\La)\cap A^-,~ \text{for}~ n\in\ZZ.
\]
We let
\[
\tP^n(\La)=\begin{cases}
\AA_n(\La)^{\times} &\text{if }n=0;\\
1+\AA_n(\La)&\text{if }n>0.
\end{cases}
\]
Then~$\tP(\La)=\tP^0(\La)$ is a compact open subgroup of~$\tG$
and~$\tP^n(\La)$,~$n>0$, is a decreasing filtration of~$\tP(\La)$ by
normal open subgroups.  If~$\La$ is self-dual
then~$P(\La)=\tP(\La)\cap G$ (resp.~$P^+(\La)=\tP(\La)\cap G^+$) is a
compact open subgroup of~$G$ (resp.~$G^+$) which has a decreasing
filtration of normal compact open subgroups~$P^n(\La)=\tP^n(\La)\cap
G$,~$n>0$.  We have a short exact sequence 
\[
1\to P^1(\La)\to P(\La)\xrightarrow{\pi} M(\La)\to 1
\]
where~$M(\La)$ is the~$k_0$-points of a reductive group~$\bbM$ defined
over~$k_0$. Let~$M^\circ (\La)$ denote the~$k_0$-points of the
connected component of~$\bbM$ and let~$P^\circ(\La)$ be the inverse
image of~$M^\circ (\La)$ under~$\pi$.  We call the
subgroups~$\tP(\La)$ of~$\tG$ and~$P^\circ(\La)$ of~$G$
\emph{parahoric subgroups}.

In fact, by~\cite{BroussousLemaire} and~\cite{Lemaire}, the
filtrations of parahoric subgroups defined here, by considering
different (self-dual) lattice sequences in the vector space~$V$,
coincide with the Moy--Prasad filtrations. 

Let~$\La$ be an~$\o_F$-lattice sequence in~$V$.  For integers~$a,b\in \mathbb{Z}$, we let~$a\La+b$ be denote the~$\o_F$-lattice sequence in~$V$ defined by
\[a\La+b(r)=\La\left(\lfloor (r-b)/a\rfloor\right),\]
for all~$r\in\mathbb{Z}$.  The \emph{affine class} of~$\La$, is the set of lattices of the form~$a\La+b$ with~$a,b\in\mathbb{Z},a\geqslant 1$.

\subsection{Semisimple strata and characters} 

A \emph{stratum} in~$A$ is a quadruple~$[\La,n,r,\b]$ where~$\La$ is
an~$\o_F$-lattice sequence in~$V$,~$n,r\in\ZZ$ with~$n\> r\> 0$,
and~$\b\in\AA_{-r}(\La)$.  A stratum~$[\La,n,r,\b]$ is called
\emph{self-dual} if~$\La$ is self-dual and~$\b\in A^{-}$.  Two
strata~$[\La,n,r,\b_1]$ and~$[\La,n,r,\b_2]$ are called equivalent
if~$\b_1-\b_2\in \AA_{-r}(\La)$.  If~$n\> r\> \frac{n}{2}>0$, an
equivalence class of strata corresponds to a character
of~$\tP_{r+1}(\La)$, by 
\[
[\La,n,r,\b]\mapsto \psi_\b
\]
where~$\psi_{\b}(x)=\psi_A(\b(x-1))$ for~$x\in\tP_{r+1}(\La)$, while
an equivalence class of self-dual strata corresponds to a character
of~$P_{r+1}(\La)$, by
\[
[\La,n,r,\b]\mapsto \psi_\b^{-}= \psi_\b\mid_{P_{r+1}(\La)}.
\]
If~$F[\b]$ is a field then we let~$B=C_A(\b)$ be the~$A$-centraliser
of~$\b$,~$\tGE=B^\times$,~$\BB_k(\La)=\AA_k(\La)\cap B$
and~$\nn_k(\b,\La)=\{x\in\AA_0(\La):\b x-x\b\in\AA_k(\La)\}$.  We
say~$[\La,n,r,\b]$ is a \emph{zero} stratum if~$n=r$ and~$\b=0$ and we
call~$[\La,n,r,\b]$ \emph{simple} if it is either zero or~$F[\b]$ is a
field,~$\La$ is an~$\o_E$-lattice sequence,~$\nu_{\La}(\b)=-n<-r$
and~$\nn_{-r}(\b,\La)\subset \BB_0(\La)+\AA_1(\La)$.

Suppose~$V=\bigoplus_{i\in I} V^i$ is a decomposition of~$V$
into~$F$-subspaces.  We let~$\La^i=\La\cap V^i$ and we
let~$\b_i=\mathbf{e}^i\b\mathbf{e}^i$, where~$\mathbf{e}^i:V\to V^i$
is the projection with kernel~$\bigoplus_{j\neq i}V^j$.  The
decomposition~$V=\bigoplus_{i\in I} V^i$ of~$V$ is called a
\emph{splitting} of~$[\La,n,r,\b]$ if~$\b=\sum_{i\in I}\b_i$
and~$\La(k)=\bigoplus_{i\in I}\La^i(k)$, for all~$k\in\ZZ$.  A
stratum~$[\La,n,r,\b]$ in~$A$ is called \emph{semisimple} if it is
zero or~$\nu_{\La}(\b)=-n$ and there exists a
splitting~$\bigoplus_{i\in I} V^i$ for~$[\La,n,r,\b]$ such that:
\begin{enumerate}
\item for~$i\in I$, the stratum~$[\La^i,q_i,r,\b_i]$ in~$\End_F(V^i)$
is simple, where
\[
q_i=\begin{cases}
r&\text{if }\b_i=0,\\
-\nu_{\La^i}(\b_i)&\text{otherwise;}
\end{cases}
\]
\item for~$i,j\in I$ with~$i\neq j$, the
stratum~$[\La^i\oplus\La^j,\max\{q_i,q_j\},r,\b_i+\b_j]$ is not
equivalent to a simple stratum in~$\End_F(V^i\oplus V^j)$.
\end{enumerate}
We write~$E=F[\b]$ and~$E_i=F[\b_i]$, hence~$E=\bigoplus_{i\in I}E_i$
is a sum of fields. As in the case when~$E$ is a field, we
write~$B=C_A(\b)$ and~$\tGE=B^\times$.  By abuse of notation, we will
call a sum~$\bigoplus_{i\in I} \La_i$ of~$\o_{E_i}$-lattice sequences
in~$V_i$ an~$\o_E$-lattice sequence in~$V$. We
write~$\BB_k(\b,\La)=\AA_k(\La)\cap B$ which gives the filtration
on~$B$ by considering~$\La$ as an~$\o_E$-lattice sequence.  We
write~$\BB(\b,\La)=\BB_0(\b,\La)$,~$\QQ(\b,\La)=\BB_1(\b,\La)$
and~$\AA(\La)=\AA_0(\La)$.

Let~$A^{ij}=\Hom_F(V^j,V^i)$ and~$\Ll=\bigoplus_{i\in I} A^{ii}$, and
write~$\tL=\Ll^\times=\prod_{i\in I} \tG_i$,
where~$\tG_i=\Aut_F(V^i)$. Also put~$B_i=C_{A^{ii}}(\b_i)$
and~$\tG_{E_i}=B^\times_i\subseteq\tG_i$.
Then~$B=\bigoplus_{i\in I} B_i\subseteq\Ll$ and~$\tGE=\prod_{i\in
  I}\tG_{E_i}\subseteq\tL$. We write~$\La_E$ when we want to make it
clear that we are considering~$\La$ as an~$\o_E$-lattice sequence.

If~$[\La,n,0,\b]$ is a non-zero semisimple stratum we let 
\[
k_0(\b,\La)=-\min\{r\in\ZZ:[\La,n,r,\b]\text{ is not semisimple}\}
\] 
denote the \emph{critical exponent} of~$[\La,n,0,\b]$
and~$k_F(\b):=\frac{1}{e(\La)}k_0(\b,\La)$; by~\cite[\S3.1]{St05}, this
is independent of~$\La$.

If~$[\La,n,r,\b]$ is self-dual with associated
splitting~$V=\bigoplus_{i\in I}V^i$ then, for each~$i\in I$, there
exists a unique~$\s(i)=j\in I$ such that~$\ov{\b_i}=-\b_j$.  We
set~$I_0=\{i\in I:\s(i)=i\}$ and choose a set of representatives~$I^+$
for the orbits of~$\s$ in~$I\setminus I_0$.  Then we
let~$I_{-}=\s(I_{+})$ so that we have a disjoint union~$I=I_+\cup
I_0\cup I_{-}$.

A semisimple stratum~$[\La,n,r,\b]$ is called \emph{skew} if it is
self-dual and the associated splitting~$\bigoplus_{i\in I} V^i$ is
orthogonal with respect to the~$\epsilon$-hermitian form~$h$,
\ie~$I=I_0$ in the notation above.  In this case, we
let~$G_{E_i}=\tG_{E_i}\cap G$ and~$\GE=\prod_{i\in I}G_{E_i}$.

Associated to a semisimple stratum~$[\La,n,r,\b]$ there are
two~$\o_F$-orders~$\HH(\b,\La)$ and~$\JJ(\b,\La)$ which are defined
inductively in~\cite[\S3.2]{St05}.  These give rise to compact open
subgroups~$\tH(\b,\La)=\HH(\b,\La)\cap\tP(\La)$
and~$\tJ(\b,\La)=\JJ(\b,\La)\cap\tP(\La)$ of~$\tG$ with decreasing
filtrations~$\tH^i(\b,\La)=\HH(\b,\La)\cap\tP_i(\La)$
and~$\tJ^i(\b,\La)=\JJ(\b,\La)\cap\tP_i(\La)$, for~$i\> 1$ by compact
open normal subgroups.

If~$[\La,n,r,\b]$ is self-dual then the associated orders and groups
are stable under the action of~$\Si$ and we
write~$\JJ^{-}(\b,\La)=\JJ(\b,\La)\cap
A^{-}$,~$J(\b,\La)=\tJ(\b,\La)\cap G$,~$J^+(\b,\La)=\tJ(\b,\La)\cap
G^+$,~$J^i(\b,\La)=\tJ^i(\b,\La)\cap G$, for~$i\> 1$, and similarly
define~$\HH^-(\b,\La),H(\b,\La),H^i(\b,\La)$. We
have~$J(\b,\La)=P(\La_E)J^1(\b,\La)$ and
\[
J(\b,\La)/J^1(\b,\La)\simeq P(\La_E)/P^1(\La_E)\simeq M(\La_E).
\]
The group~$M(\La_E)$ is the group of points of a finite reductive
group over~$k_F$, and we denote by~$J^\circ(\b,\La)$ the inverse image
of the connected component~$M^\circ(\La_E)$ under the projection map.

By~\cite[Proposition~3.4]{St05}, the stratum~$[\La,n,r+1,\b]$ is
equivalent to a semisimple stratum~$[\La,n,r+1,\g]$
with~$\g\in\Ll$. In~\cite[Definition~3.13]{St05}, for~$0\< m <r+1$, a
set of characters~$\Cc(\La,m,\b)$ of~$\tH^{m+1}(\b,\La)$ is attached
to~$[\La,n,r,\b]$, depending on our initial choice of~$\psi_F$. Precisely,~$\Cc(\La,m,\b)$ consists of the
characters~$\tth$ of~$\tH^{m+1}(\b,\La)$ which satisfy
\begin{enumerate}
\item~$\tth\mid_{\tH^{m+1}(\b,\La)\cap \tG_{i}}$ is a simple
  character, in the sense of~\cite[Definition~3.2.3]{BK93};
\item if~$m'=\max\{m,\lceil r/2\rceil\}$ then there
  exists~$\tth_0\in\Cc(\La,m',\g)$ such
  that~$\tth\mid_{\tH^{m'+1}(\b,\La)}=\tth_0\psi_{\b-\g}$. 
\end{enumerate}
If~$[\La,n,r,\b]$ is self-dual then~$\Cc(\La,m,\b)$ is preserved by
the involution~$\s$ and, as in~\cite[\S~$3.6$]{St05}, one
associates to~$[\La,n,r,\b]$ the set~$\Cc_{-}(\La,m,\b)$ of characters
of~$H^{m+1}(\b,\La)$ obtained by restriction from~$\Cc(\La,m,\b)^\Si$.

The following results were proved in the case~$R=\CC$ but, since the
groups involved are all pro-$p$, their proofs apply provided the
characteristic of~$R$ is not~$p$, as is the case here.

\begin{thm}[{\cite[Theorem~3.22]{St05}}]\label{theorem22}
Let~$[\La,n,0,\b]$ be a semisimple stratum in~$A$.  
\begin{enumerate}
\item\label{theorem221}  If~$\tth\in\Cc(\La,0,\b)$ then~$I_{\tG}(\tth)=\tJ^1(\b,\La)\tGE\tJ^1(\b,\La)$.
\item\label{theorem222} Let~$[\La',n',0,\b]$ be another semisimple stratum in~$A$.
  There is a bijection
\[
\t_{\La,\La',\b}:\Cc(\La,0,\b)\to\Cc(\La',0,\b),
\]
called the transfer map, which takes~$\tth\in\Cc(\La,0,\b)$ to the
unique character~$\tth'\in\Cc(\La',0,\b)$ such that~$\tGE\subseteq
I_{\tG}(\tth,\tth')$.
\end{enumerate}
\end{thm}

Let~$[\La,n,r,\beta]$ be a semisimple stratum.  The \emph{affine class} of~$[\La,n,r,\beta]$ is the set of all (semisimple) strata of the form
\[[\La',n',r',\beta],\]
where~$\La'=a\La+b$ is in the affine class of~$\La$,~$n'=an$ and~$r'$ is any integer such that~$\lfloor r'/a\rfloor=r$.  By induction on~$k_F(\beta)$ (cf.~\cite[Lemma 2.2]{BSS}), many objects associated to a semisimple stratum only depend on the affine class of the stratum.  In particular, if~$[\La',n',r',\beta]$ is in the affine class of~$[\La,n,r,\beta]$, we have:
\begin{enumerate}
\item~$\tH^{m'+1}(\beta',\La')=\tH^{m+1}(\beta,\La)$;
\item~$\mathcal{C}(\La',m',\beta')=\mathcal{C}(\La,m,\beta)$;
\item the transfer map~$\tau_{\La,\La',\beta}:\mathcal{C}(\La,m,\beta)\rightarrow \mathcal{C}(\La',m',\beta)$ is the identity.
\end{enumerate}

If the associated strata are self-dual, then we have the following
analogue of Theorem~\ref{theorem22}.

\begin{thm}[{\cite[Lemma~2.5]{MiSt}}]\label{theorem23}
Let~$[\La,n,0,\b]$ be a self-dual semisimple stratum in~$A$.  
\begin{enumerate}
\item If~$\th\in\Cc_{-}(\La,0,\b)$ then~$I_{G}(\th)=J^1(\b,\La)\GE J^1(\b,\La)$.
\item Let~$[\La',n',0,\b]$ be another self-dual semisimple stratum
  in~$A$.  There is a bijection 
\[
\t_{\La,\La',\b}:\Cc_{-}(\La,0,\b)\to\Cc_{-}(\La',0,\b),
\]
called the transfer map, which takes~$\th\in\Cc_{-}(\La,0,\b)$ to the
unique character~$\th'\in\Cc_{-}(\La',0,\b)$ such that~$\GE\subseteq
I_{G}(\th,\th')$.
\end{enumerate}
\end{thm}

Let~$[\La,n,0,\b]$ be a semisimple stratum and~$\tth\in\Cc(\La,0,\b)$.  

\begin{thm}[{\cite[Corollary~3.25]{St05}}]\label{theorem32}
There exists a unique irreducible representation~$\teta$ of~$\tJ^1(\b,\La)$
containing~$\tth$.
\end{thm}

If~$[\La,n,0,\b]$ is self-dual and~$\th\in\Cc_{-}(\La,0,\b)$, then we
have the following analogue of Theorem~\ref{theorem32}.

\begin{thm}[{\cite[Lemma~2.5]{MiSt}}]\label{theorem33}
There exists a unique representation~$\eta$ of~$J^1(\b,\La)$ containing~$\th$. 
\end{thm}

We call the representations~$\eta$ and~$\teta$ of
Theorems~\ref{theorem32} and~\ref{theorem33}, \emph{Heisenberg
  representations}.  We define a bijection, which we also denote
by~$\t_{\La,\La',\b}$, between the set of Heisenberg representations
of~$\tJ^1(\b,\La)$ containing a semisimple character
in~$\Cc(\La,0,\b)$ and the set of Heisenberg representations
of~$\tJ^1(\b,\La')$ containing a semisimple character
in~$\Cc(\La',0,\b)$ which restricts to the transfer map,
\ie~if~$\teta$ is the unique Heisenberg representation
of~$\tJ^1(\b,\La)$ containing~$\tth\in\Cc(\La,0,\b)$
then~$\t_{\La,\La',\b}(\teta)$ is the unique Heisenberg
representation of~$\tJ^1(\b,\La')$
containing~$\t_{\La,\La',\b}(\tth)$.   Similarly, we define a
bijection~$\t_{\La,\La',\b}$ between the set of Heisenberg
representations of~$J^1(\b,\La)$ containing a self-dual semisimple
character in~$\Cc_{-}(\La,0,\b)$ and the set of Heisenberg
representations of~$J^1(\b,\La')$ containing a self-dual semisimple
character in~$\Cc_{-}(\La',0,\b)$.

\subsection{Double coset identities}

We state mild generalisations of some results of~\cite{StDC}, the
proofs of which, [\emph{op.\,cit.}, Lemmas~$2.1$,~$2.2$ and
Theorem~$2.3$], still apply.  The notation in this short subsection is
independent of that in the rest of the paper. 
Let~$G$ be a group and~$\Ga$ a group of
automorphisms of~$G$.  If~$H$ is a~$\Ga$-stable subgroup of~$G$ we
let~$H^\Ga$ denote subgroup of fixed points of~$\Ga$.

\begin{thm}\label{fixedpt1}Let~$U_1$ and~$U_2$ be~$\Ga$-stable
subgroups of~$G$.
\begin{enumerate}
\item\label{fixedpt1:1}
Suppose that, for all~$g\in G$, the (non-abelian) cohomology
pointed set~$H^1(\Ga,gU_1g^{-1}\cap U_2)$ is trivial.  Then, for all~$g\in
G^{\Ga}$, we have~$(U_1gU_2)^\Ga=U_1^\Ga gU_2^\Ga$.
\item Suppose that~$\Ga$ is a soluble group of order coprime to~$p$,
that~$U_1$ and~$U_2$ are~$\Ga$-stable pro-$p$ subgroups of~$G$, and
that~$g\in G$. 
\begin{enumerate}
\item~$(U_1gU_2)^\Ga\neq \emptyset$ if and only if~$U_1gU_2$ is stable
under~$\Ga$.
\item Let~$H$ be a~$\Ga$-stable subgroup of~$G$ such that~$U_1hU_2\cap
  H=(U_1\cap H)h(U_2\cap H)$, for all~$h\in H$.
  Then~$(U_1HU_2)^\Ga=U_1^\Ga H^\Ga U_2^\Ga$.
\end{enumerate}
\end{enumerate}
\end{thm}

\subsection{Modular representation theory techniques}

As~$R$-representations of compact open subgroups are not necessarily
semisimple (unlike the case~$R=\CC$), we will need to use appropriate
versions of some well known representation theory techniques.  The
first is the \emph{simple criterion for irreducibility}
of~\cite{vigneras}.

\begin{lemma}\label{simpcrit1}
Let~$\l$ be an irreducible representation of a compact open
subgroup~$K$ of~$G$. Suppose that~$\End_G(\ind_K^G(\l))\simeq R$ and,
for any irreducible representation~$\pi$ of~$G$, if~$\l$ is a
subrepresentation of~$\pi$ then it is also a quotient of~$\pi$.
Then~$\ind_K^G(\l)$ is irreducible.
\end{lemma}

A representation~$\pi$ of~$G$ is called \emph{quasi-projective} if,
for all representations~$\pi'$ of~$G$ and all surjective
homomorphisms~$\varphi:\pi\to \pi'$, the homomorphism~$\End_G(\pi)\to
\Hom_G(\pi,\pi')$,~$\a\mapsto \a\circ \varphi$ for~$\a\in\End_G(\pi)$,
is surjective.  The second modular representation theory criterion we
make use of is the \emph{simple criterion for quasi-projectivity}
of~\cite{vigneras}
(\cf~also~\cite[Proposition~$3.15$]{HenniartSecherre}).
 
\begin{lemma}\label{quasiprojlemma}
Let~$K$ be a compact open subgroup of~$G$,~$\l$ an irreducible
representation of~$K$ and~$\pi=\ind_K^G(\l)$.  If the~$\l$-isotypic
component of~$\pi$ is a direct summand of the restriction of~$\pi$
to~$K$ and no subquotient of its complement is isomorphic to~$\l$
then~$\pi$ is quasi-projective.
\end{lemma}

Let~$\pi,\t$ be~$R$-representations of~$G$. Then~$\Hom_G(\pi,\t)$
is a right~$\End_G(\pi)$-module by pre-composition.  In attempts to
classify the irreducible representations of~$G$, quasi-projective
representations are particularly interesting due to the following
theorem of Arabia.

\begin{thm}[{\cite[Appendix Th\'eor\`eme~10]{vignerasselecta}}]\label{quasiprojvigneras}
Suppose~$\pi$ is quasi-projective and finitely generated. Then the
functor~$\RR_R(G)\to \End_G(\pi)\text{-mod}$,~$\t\mapsto
\Hom_G(\pi,\t)$, induces a bijection between the irreducible
quotients of~$\pi$ and the simple right~$\End_G(\pi)$-modules. 
\end{thm}
%
%
%

Suppose that~$J$ is a compact open subgroup of~$G$ containing
a compact open pro-$p$ subgroup~$J^1$ which is normal in~$J$ and
that~$\eta$ is an irreducible representation of~$J^1$ which extends to
an irreducible representation~$\k$ of~$J$.  Then we have the following
lemma, implicit in~\cite{vigneras} (\cf~\cite[Proposition~4.2]{VignerasLanglands2} and~\cite[Lemme~2.6]{VA2} for a proof).

\begin{lemma}\label{simpcrit3}
The functor~$\k\otimes -$ induces an equivalence of categories
between~$\RR_R(J/J^1)$ and the category~$\RR_R(J,\eta)$
of~$\eta$-isotypic representations of~$J$.
\end{lemma}

The following lemma is a mild abstraction of~\cite[Proposition~5.3.2]{BK93}.
\begin{lemma}\label{BK532}
Let~$X_1$ and~$X_2$ be subgroups of~$G$, and~$X_1^1$ (resp.~$X_2^1$) be a subgroup of~$X_1$ (resp.~$X_2$).  For~$i=1,2$, let~$\zeta_i$ be a representation of~$X_i$ trivial on~$X_i^1$, and let~$\mu_i$ be a representation of~$X_i$.  Suppose that
\[\Hom_{X_1\cap X_2}(\mu_1,\mu_2)=\Hom_{X^1_1\cap X^1_2}(\mu_1,\mu_2)\simeq R.\]
Then, for any non-zero~$S\in \Hom_{X_1\cap X_2}(\mu_1,\mu_2)$, the map
\begin{center}
\begin{tikzpicture}[>=stealth]
\matrix [matrix of math nodes, column sep=1cm, row sep=0.2cm,text height=1.5ex, text depth=0.25ex, minimum width=14em]
{
\node(row11) {\Hom_{X_1\cap X_2}(\zeta_1, \zeta_2)}; &\node(row12) {\Hom_{X_1\cap X_2}(\mu_1\otimes\zeta_1,\mu_2\otimes\zeta_2)}; \\
\node(row21) {T}; &\node(row22) {S\otimes T}; \\
};
\draw  [->]  (row11) -- (row12);
\draw  [|->]  (row21) -- (row22);
\end{tikzpicture}
\end{center}
is an isomorphism of vector spaces.  
\end{lemma}

\begin{proof} 
It is easy to check that the map is well-defined, and it is clearly
injective, so we need only check surjectivity.
Let~$f\in \Hom_{X_1\cap X_2}(\mu_1\otimes\zeta_1,\mu_2\otimes\zeta_2)$ be non-zero.
Write~$f$ as a finite sum~$\sum_k S_k\otimes T_k$, with~$S_k\in\Hom_R(\mu_1,\mu_2)$ non-zero and~$T_k\in \Hom_R(\zeta_1,\zeta_2)$, such that~$\{T_k\}$ is linearly independent over~$R$. Let~$x\in X_1^1\cap X_2^1$; then~$f\circ \mu_1\otimes\zeta_1(x)=\mu_2\otimes\zeta_2(x)\circ f$.  Hence, as~$\zeta_1,\zeta_2$
are trivial on~$X_1^1\cap X_2^1$, we have
\[
\sum_{k}(S_k\mu_1(y)-\mu_2(y)S_k)\otimes T_k=0,
\]
for~$y\in X_1^1\cap X_2^1$.
Thus~$S_k\in\Hom_{X_1^1\cap X_2^1}(\mu_1,\mu_2)$, by the linear independence of~$\{T_k\}$.  The intertwining
spaces~$\Hom_{X_1^1\cap X_2^1}(\mu_1,\mu_2)$
and~$\Hom_{X_1\cap X_2}(\mu_1,\mu_2)$ are
one-dimensional and equal by our hypotheses.  Thus~$S_k$ is a scalar multiple of~$S$ and we can
write~$f=S\otimes T$ with~$T\in \Hom_R(\zeta_1,\zeta_2)$.  Furthermore,
\[
S\otimes T(\mu_1\otimes\zeta_1(y)v)=(\mu_2(y) S\otimes \zeta_2(y)T)(v)
\]
and
\[
S\otimes T(\mu_1\otimes\zeta_1(y)v)=(S\mu_1(y)\otimes T\zeta_1(y))(v)
=(\mu_2(y) S\otimes T\zeta_1(y))(v)
\]
for all~$y\in X_1\cap X_2$ and~$v$ in the space of~$\mu_1\otimes \zeta_1$.
Hence~$T\in\Hom_{X_1\cap X_2}(\zeta_1,\zeta_2)$ and, since~$f=S\otimes T$, our map is surjective. 
\end{proof}
\section{Asymmetric generalisations via~$\dagger$-constructions}

In this section we present a particularly useful construction: to
an~$\o_F$-lattice sequence~$\La$ in~$V$, we
associate a strict~$\o_F$-lattice sequence~$\La^\dag$ of
period~$e(\La)$ in a direct sum of~$e(\La)$ copies of~$V$, whose
associated hereditary order~$\AA(\La^\dag)$ is principal and such that
all the blocks~$\AA^{ii}(\La^\dag)=\AA(\La)$, for~$0\< i\<
e(\La)$. This construction becomes useful later when applied to two~$\o_F$-lattice sequences~$\La$ and~$\Y$ in~$V$, which, if necessary, after changing in their affine classes we assume~$e(\La)=e(\Y)$; in
this situation~$\AA(\La^\dag)$ and~$\AA(\Y^\dag)$ are principal orders
in~$V^\dag$ of the same block size, hence are conjugate, yet when we
restrict to a single block we find the not necessarily conjugate
orders~$\AA(\La)$ and~$\AA(\Y)$.   This construction originates in work of the second author with Broussous and S\'echerre in~\cite{BSS}.  The first part of this section is concerned with revisiting the construction of~[ibid.] and generalising it to semisimple strata.  Then we provide two new applications of~$\dag$: a generalisation of the semisimple intersection property of~\cite{St08} and an extension of the computation of the intertwining a semisimple character in~\cite{St05} to the case of two semisimple characters related by transfer.

\subsection{The~$\dagger$-construction}\label{daggerconst}

Let~$\La$ be an~$\o_F$-lattice sequence in~$V$ of~$\o_F$-period~$e(\La)$. Let~$V^{\dag}=V\oplus \cdots\oplus V$ ($e(\La)$ times).  Following~\cite[Section 2]{BSS}, we define an~$\o_F$-lattice sequence~$\La^\dag$ in~$V^\dag$ by
\[
\La^{\dag}(r)=\bigoplus_{k=0}^{e(\La)-1}\La(r+k),\text{ for all~$r\in\ZZ$.}
\]
Then, for all~$r\in \mathbb{Z}$,
\[
\dim_{k_F}(\La^{\dag}(r)/\La^{\dag}(r+1))=\sum_{k=0}^{e(\La)-1}\dim_{k_F}(\La(r+k)/\La(r+k+1))
=\dim_F(V).
\]
Therefore,~$\La^{\dag}$ is a strict~$\o_F$-lattice sequence
in~$V^{\dag}$ of period~$e(\La)$ whose associated order~$\AA(\La^\dag)$ is
principal.

Let~$[\La,n,r,\b]$ be a semisimple stratum in~$A$ with associated splitting~$V=\bigoplus_{i\in I}V^i$, and~$e=e(\La)=e(\La_i)$.  For each~$i\in I$, let~$V^{i,\dag}=V^i\oplus \cdots\oplus V^i$ ($e(\La)$ times), and let~$\La_i^{\dag}$ be the~$\o_F$-lattice sequence in~$V^{i,\dag}$, defined as above.  Let~$V^\dag=\bigoplus_{i\in I}V^{i,\dag}$ and let~$\La^\dag$ be the~$\o_F$-lattice sequence in~$V'$ defined by~$\La^\dag=\bigoplus_{i\in I}\La^{i,\dag}$. Note that this is the same lattice sequence as that defined above (working directly with~$\La$ within~$V$). Let~$A^\dag=\End_F(V^\dag)$ and~$\tG^{\ddag}=\Aut_F(V^\dag)$.  

We recall that~$\beta=\sum_{i\in I}\beta_i$, where~$\beta_i=e_i\beta e_i$ and~$e_i:V\rightarrow V^i$ is the projection map with kernel~$\bigoplus_{j\neq i}V^j$.  Let~$\beta_i^\dag$ denote the image of~$\beta_i$ under the diagonal embedding of~$\End_F(V^i)$ into~$\End_F(V^{i,\dag})$, and~$\beta^\dag=\sum_{i\in I}\beta_i^\dag$.  Then~$\La^{i,\dag}$ is an~$\o_{E_i}$-lattice sequence, whose associated hereditary~$\o_F$-order~$\mathfrak{A}(\La^{i,\dag})$ is principal. Moreover, the stratum~$[\La^\dag,n,r,\b^\dag]$ in~$A^\dag$ is semisimple, with associated splitting~$V^\dag=\bigoplus_{i\in I}V^{i,\dag}$.

We recall also that~$\tL$ is the stabilizer in~$\tG$ of the decomposition~$V=\bigoplus_{i\in I} V^i$. Let~$\tQ=\tL\tU^+_Q$ be a parabolic subgroup of~$\tG$ with Levi component~$\tL$, and opposite parabolic~$\tQ^-=\tL\tU^-_Q$ with respect to~$\tL$. Then, for any~$m\ge 0$, the group~$\tH^{m+1}(\b,\La)$ has an Iwahori decomposition with respect to~$(\tL,\tQ)$ with
\begin{equation}\label{eqn:thetadecomp}
\tH^{m+1}(\b,\La)\cap\tL = \prod_{i\in I} \tH^{m+1}(\b_i,\La^i).
\end{equation}
Moreover, by~\cite[Lemma~3.15]{St05}, any semisimple character~$\tth\in\Cc(\b,m,\La)$ is trivial on the unipotent parts~$\tH^{m+1}(\b,\La)\cap\tU^{\pm}_Q$ and
\[
\tth|_{(\tH^{m+1}(\b,\La)\cap\tL)}=\bigotimes_{i\in I} \tth_i,
\]
with~$\tth_i\in\Cc(\b_i,m,\La^i)$ a simple character. Analogously, we have the Levi subgroup~$\tL^\dag$ which is the stabilizer of the decomposition~$V^\dag=\bigoplus_{i\in I} V^{i,\dag}$ and~$\tH^{m+1}(\b^\dag,\La^\dag)$ has an Iwahori decomposition with respect to any parabolic subgroup~$\tQ^\dag$ with Levi component~$\tL^\dag$, with
\[
\tH^{m+1}(\b^\dag,\La^\dag)\cap\tL^\dag = \prod_{i\in I} \tH^{m+1}(\b_i^\dag,\La^{i,\dag}).
\]

Let~$\Mm^\dag$ denote the Levi subalgebra of~$A^\dag$ which is the stabilizer of the splitting~$V^\dag=V\oplus\cdots\oplus V$, and let~$\tM^\ddag$ be its group of units. Let~$\Ga$ be the subgroup of~$\tM^\ddag$ consisting of elements with blocks~$\pm\Id$.  Let~$\tP^\ddag$ be any parabolic subgroup of~$\tG^\ddag$ with Levi factor~$\tM^\ddag$ and unipotent radical~$\tU^\ddag$, and let~$\tP^{-,\dag}$ denote the opposite parabolic of~$\tP^{\ddag}$ with respect to~$\tM^\ddag$, with Levi decomposition~$\tP^{-,\dag}=\tM^\ddag\ltimes\tU^{-,\dag}$. Similarly, for each~$i\in I$, we have a Levi subgroup~$\tM^{i,\dag}$ of~$\tG_i^\dag=\Aut_F(V^{i,\dag})$.

For all~$m\>0$, using~\cite[Proposition~5.2]{St08}, we have an Iwahori
decomposition
\begin{align}
\label{eqn:HIwa}
\tH^{m+1}(\b^\dag,\La^\dag)=(\tH^{m+1}(\b^\dag,\La^\dag)&\cap \widetilde{U}^{-,\dag})(\tH^{m+1}(\b^\dag,\La^\dag)\cap \tM^\ddag)(\tH^{m+1}(\b^\dag,\La^\dag)\cap \tU^\ddag),\\[5pt]
\notag \tH^{m+1}(\b^\dag,\La^\dag)\cap \tM^\ddag&=\tH^{m+1}(\b,\La)\times\cdots\times\tH^{m+1}(\b,\La).
\end{align}
There are similar decompositions for~$\tH^{m+1}(\b_i^\dag,\La^{i,\dag})$.

Let~$\tth\in\Cc(\b,m,\La)$ be a semisimple character, corresponding to simple characters~$\tth_i\in\Cc(\b_i,m,\La^i)$ as in~\eqref{eqn:thetadecomp}. Put~$\tth_i^\dag=\tau_{\Lambda^i,\La^{i,\dag},\beta_i,\beta_i^\dag}(\tth_i)$, the transfer of~$\tth_i$ to~$\Cc(\b_i^\dag,m,\La_i^\dag)$. By~\cite[Lemma 2.7]{BSS}, the restriction of~$\tth_i^\dag$ to~$\tH^{m+1}(\b_i^\dag,\La^{i,\dag})\cap\tM^{i,\dag}$ has the form~$\tth_i\otimes\cdots\otimes\tth_i$; moreover, for~$\tP^{i,\dag}=\tM^{i,\dag}\tU^{i,\dag}$ any parabolic subgroup of~$\tG^{i,\dag}$ with Levi component~$\tM^{i,\dag}$, the restriction of~$\tth_i^\dag$ to the unipotent part~$\tH^{m+1}(\b_i^\dag,\La^{i,\dag})\cap\tU^{i,\dag}$ is trivial.

\begin{lemma}\label{descthetadag}
There is a unique semisimple character~$\tth^\dag\in\Cc(\b^\dag,m,\La^\dag)$ such that
\[
\tth^\dag|_{(\tH^{m+1}(\b^\dag,\La^\dag)\cap\tL^\dag)} = \bigoplus_{i\in I}\tth_i^\dag;
\]
Moreover,~$\tth^{\ddag}$ is trivial on the unipotent parts
in~\eqref{eqn:HIwa}, and
\[
\tth^\dag\mid_{(\tH^{m+1}(\b^\dag,\La^\dag)\cap\tM^{\ddag})}=\tth\otimes\cdots\otimes\tth.
\]
\end{lemma}

\begin{proof}
The first part follows easily from the inductive definition of semisimple characters (see in particular~\cite[Lemma~3.15]{St05}). Moreover, for any parabolic subgroup~$\tQ^\dag=\tL^\dag\tU_Q^\dag$ with Levi component~$\tL^\dag$, the restriction of~$\tth^\dag$ to~$\tH^{m+1}(\b^\dag,\La^\dag)\cap\tU_Q^\dag$ is trivial; the second statement follows from this, the corresponding statement in the simple case (\cite[Lemma 2.7]{BSS}) and the unicity in~\cite[Lemma~3.15]{St05} again.
\end{proof}

For~$g\in \tG$, let~$g^\dag$ denote its diagonal embedding in~$\tG^\ddag$.
\begin{lemma}\label{dagints}
For~$i=1,2$, let~$\tth_i$ be semisimple characters in~$\mathcal{C}(\La,m,\beta_i)$.  If~$g$ intertwines~$\tth_1$ and~$\tth_2$, then~$g^\dag$ intertwines~$\tth_1^\dag$ and~$\tth_2^\dag$.
\end{lemma}

\begin{proof}
For simple characters, it is shown in the proof of~\cite[Proposition~2.6]{BSS} that this follows from~\cite[Lemma~2.7]{BSS}.  The proof in the semisimple case follows  \emph{mutatis mutandis} using Lemma~\ref{descthetadag} in place of~\cite[Lemma~2.7]{BSS}.
\end{proof}
\subsection{Applications of~$\dagger$}


Let~$[\La,n_{\La},0,\b]$ and~$[\Y,n_{\Y},0,\b]$ be semisimple strata
in~$A$ with splitting~$V=\bigoplus_{i\in I}V^i$.  
 Let~$e_{\La}$
(resp.~$e_{\Y}$) denote the~$\o_F$-period of~$\La$ (resp.~$\Y$), and
hence of~$\La^i$ (resp.~$\Y^i$) for all~$i\in I$.  By changing~$[\La,n_{\La},0,\b]$ and~$[\Y,n_{\Y},0,\b]$ in their affine classes, we assume the~$e=e_{\La}=e_{\Y}$.  As remarked earlier, this does not change the objects (orders, groups, characters) associated to the semisimple strata.

For~$i\in I$, we apply the construction of Section~\ref{daggerconst} to~$\La^{i}$ and to~$\Y^i$.   Suppose that the~$\o_{E_i}$-period~$e_{E_i}$ of~$\La$, and hence of~$\Y$, is related to the~$\o_F$-period~$e$, by 
\[e_{E_i}=m_ie,\]
so that~$m_i$ is the ramification index of~$E_i/F$. Then, for all~$r\in\mathbb{Z}$,
\[
\dim_{k_{E_i}}(\La^{i,\dag}(r)/\La^{i,\dag}(r+1))=\sum_{k=0}^{m_ie_{E_i}-1}\dim_{k_{E_i}}(\La^i(r+k)/\La^i(r+k+1))
=m_i\dim_{E_i}(V^i).
\]
Hence, the lattice sequences~$\La^{i,\dag}$ and~$\Y^{i,\dag}$ are strict~$\o_{E_i}$-lattice sequences in~$V^{i,\dag}$ of~$\o_F$-period~$e$ (and~$\o_{E_i}$-period~$e_{E_i}$).  Furthermore, the associated hereditary~$\o_{E_i}$-orders~$\BB(\b_i,\La^{i,\dag})$and~$\BB(\b_i,\Y^{i,\dag})$ are principal~$\o_{E_i}$-orders with the same block size, hence there exist~$x_i\in C_{\tG^{\ddag}_i}(\b_i^\dag)$, such that
\[
\La^{i\ddag}(r)=x_i\cdot \Y^{i\ddag}(r),
\]
for all~$r\in\ZZ$.  Let~$x=\sum_{i\in I}x_i$; then~$x\in \tGE^\ddag$
and we have
\[
\La^{\dag}=x\cdot \Y^{\dag}.
\]
It follows that the data coming from the semisimple
strata~$[\La^\dag,n_{\La},0,\b^\dag]$ and~$[\Y^\dag,n_{\Y},0,\b^\dag]$
are conjugate in~$\tGE^\ddag$ and we get: 

\begin{lemma}\label{lem:conj}
In the situation above, there exists~$x\in \tGE^\ddag$ such that
\begin{enumerate}
\item~$\JJ(\b^\dag,\La^\dag)= \JJ(\b^\dag,\Y^\dag)^x$ and~$\HH(\b^\dag,\La^\dag)= \HH(\b^\dag,\Y^\dag)^x$;
\item~$\tJ(\b^\dag,\La^\dag)= \tJ(\b^\dag,\Y^\dag)^x$ and~$\tH(\b^\dag,\La^\dag)= \tH(\b^\dag,\Y^\dag)^x$;
\item conjugation by~$x$ defines a bijection~$\Cc(\b^\dag,0,\La^\dag)\to \Cc(\b^\dag,0,\Y^\dag)$.
\end{enumerate}
\end{lemma}

Throughout this section, ``applying the~$\dagger$-construction'' will mean applying it in the way just described.

\subsection{Semisimple intersection property}

In this section we generalise the semisimple intersection property
of~\cite[Lemma~$2.6$]{St08}.

\begin{lemma}\label{GLssintprop}
Let~$[\La,n_{\La},0,\b]$ and~$[\Y,n_{\Y},0,\b]$ be semisimple strata
in~$A$ and~$y\in \tGE$.  Then
\[
\tP^1(\Y)y\tP^1(\La)\cap \tGE=\tP^1(\Y_{E})y\tP^1(\La_{E}).
\]
\end{lemma}

\begin{proof}
Applying the~$\dagger$-construction, by Lemma~\ref{lem:conj}
we have~$x\in\tGE^\ddag$ such that
\[
\tP^1(\Y^\dag)y^\dag \tP^1(\La^\dag)=\tP^1(\Y^\dag)y^\dag x\tP^1(\Y^\dag)x^{-1}.
\]
By the semisimple intersection property in~$\tGE^\ddag$ (\cf~the proof
of~\cite[Lemma~$2.6$]{St08}), because~$x\in\tGE^\ddag$ we have
\[
\tP^1(\Y^\dag)y^\dag x\tP^1(\Y^\dag)\cap\tGE^\ddag=\tP^1(\Y_{E}^\dag)y^\dag x\tP^1(\Y_{E}^\dag).
\]
Hence
\[
\tP^1(\Y^\dag)y^\dag \tP^1(\La^\dag)\cap \tGE^\ddag=\tP^1(\Y_{E}^\dag)y^\dag \tP^1(\La_{E}^\dag).
\]
Recall,~$\tM^\ddag$  is the Levi subgroup of~$\tG^\ddag$ defined by the decomposition of~$V^\dag$ into a sum of copies of~$V$, and~$\Ga$ is the~$2$-subgroup of~$\tM^\ddag$ consisting of elements with blocks~$\pm \Id$.  Notice that,~$\tM^\ddag$ is equal to the fixed point subgroup of~$\tG^\ddag$ under the conjugation action of~$\Ga$. Hence, because~$\Ga$ is a~$2$-group and~$\tP^1(\La_{E}^\dag)$
and~$\tP^1(\Y_{E}^\dag)$ are pro-$p$ groups, with~$p$
odd,~$H^1(\Ga,y^\dag\tP^1(\Y_{E}^\dag)(y^\dag)^{-1}\cap
\tP^1(\La_{E}^\dag))=1$ and we can apply
Theorem~\ref{fixedpt1}\ref{fixedpt1:1} to find 
\[
\tP^1(\Y_{E}^\dag)y^\dag \tP^1(\La_{E}^\dag)\cap \tM^\ddag=
(\tP^1(\Y_{E}^\dag)\cap \tM^\ddag)y^\dag (\tP^1(\La_{E}^\dag)\cap \tM^\ddag).
\]
We have~$(\tP^1(\Y_{E}^\dag)\cap\tM^\ddag)=\prod_{i=1}^{d}
\tP^1(\Y_{E})$ and~$(\tP^1(\La_{E}^\dag)\cap\tM^\ddag)=\prod_{i=1}^{d}
\tP^1(\La_{E})$.  Thus, restricting to a single block in~$\tM^\ddag$ we
recover the result.
\end{proof}

\begin{corollary}\label{ssintprop}
Let~$[\La,n_{\La},0,\b]$ and~$[\Y,n_{\Y},0,\b]$ be self-dual
semisimple strata in~$A$.  Then
\begin{align*}
&P^1(\Y)yP^1(\La)\cap \GE^+=P^1(\Y_{E})yP^1(\La_{E}), &&\text{for }y\in \GE^+;\\
&P^1(\Y)yP^1(\La)\cap \GE=P^1(\Y_{E})yP^1(\La_{E}), &&\text{for }y\in \GE.
\end{align*}
\end{corollary}

\begin{proof}
Applying Theorem~\ref{fixedpt1}\ref{fixedpt1:1}, under the fixed
points of the involution~$\s$, we have
\[
\tP^1(\Y_{E})y\tP^1(\La_{E})\cap \GE^+=
 (\tP^1(\Y_{E})\cap\GE^+)y(\tP^1(\La_{E})\cap\GE^+).
\]
Therefore, by Lemma~\ref{GLssintprop},~$P^1(\Y)yP^1(\La)\cap
\GE^+=P^1(\Y_{E})yP^1(\La_{E}).$  The second equality follows by
intersecting with~$G$, since~$P^1(\La_E)\subseteq\GE$.
\end{proof}

A simple application of the semisimple intersection property gives us the following bijection of double cosets, where we note that~$J_{\Y}\GE J_{\La}=J^1_{\Y}\GE J^1_{\La}$.

\begin{lemma}\label{doublecosets}
Let~$[\La,n_{\La},0,\b]$ and~$[\Y,n_{\Y},0,\b]$ be self-dual semisimple strata in~$A$.  Let~$J_{\Y}=J(\b,\La)$
and~$J_{\Y}=J(\b,\Y)$.  The following map is a bijection
\begin{center}
\begin{tikzpicture}[>=stealth]
\matrix [matrix of math nodes, column sep=1cm, row sep=0.2cm,text height=1.5ex, text depth=0.25ex, minimum width=9em]
{
\node(row11) {P(\Y_E)\backslash \GE/P(\La_E)}; &\node(row12) {J_{\Y}\backslash J_{\Y}\GE J_{\La}/J_{\La}}; \\
\node(row21) {X}; &\node(row22) {J_{\Y} X J_{\La}.}; \\
};
\draw  [->]  (row11) -- (row12);
\draw  [|->]  (row21) -- (row22);
\end{tikzpicture}
\end{center}
\end{lemma}

\begin{proof}
Let~$g\in\GE$.  Considering~$\La$ and~$\Y$ as~$\o_F$-lattice sequences, we have containments~$J^1_{\Y}\subseteq P^1(\Y)$ and~$J^1_{\La}\subseteq P^1(\La)$.  Hence
\[
J_{\Y}^1(P(\Y_E) g P(\La_E))J_{\La}^1\cap \GE\subseteq 
P^1(\Y)(P(\Y_E) gP(\La_E))P^1(\La)\cap \GE.
\]
We choose a set of representatives for the finite double coset
space~$P^1(\Y)\backslash(P(\Y_E) gP(\La_E))/P^1(\La)$ and for each
representative we apply the simple intersection property,
Corollary~\ref{ssintprop}, to find
\[
P^1(\Y)(P(\Y_E) gP(\La_E))P^1(\La)\cap \GE=P(\Y_E) g P(\La_E).
\]
Therefore~$P(\Y_E) g P(\La_E)=J^1_{\Y}(P(\Y_E) g
P(\La_E))J^1_{\La}\cap \GE$ and the map is a bijection.
\end{proof}

\subsection{Intertwining of transfers}

Let~$[\La,n_{\La},0,\b]$ and~$[\Y,n_{\Y},0,\b]$ be semisimple
strata. Let~$\tth_{\La}\in\Cc(\Y,0,\b)$
and~$\tth_{\Y}=\t_{\La,\Y,\b}(\tth_{\Y})$.
We apply the~$\dag$-construction and abbreviate~$\tJ^1_{\La}=\tJ^1(\b,\La)$
and~$(\tJ_{\La}^\ddag)^1=\tJ^1(\b^\dag,\La^\dag)$, with similar notation
for~$\Y$, and also write~$\t=\t_{\La,\Y,\b}$
and~$\t^\dag=\t_{\La^\dag,\Y^\dag,\b^\dag}$.

\begin{thm}\label{intchars1}
We have
\[
I_{\tG}(\tth_{\La},\tth_{\Y})=\tJ^1_{\Y}\tGE \tJ^1_{\La}.
\]
\end{thm}

\begin{proof} 
Let~$g\in I_{\tG}(\tth_{\La},\tth_{\Y})$ and, as before, let~$g^\dag$ denote
the diagonal embedding of~$g$ in~$\tG^\ddag$. By Lemma~\ref{dagints}, we have~
\[g^\dag\in I_{\tG}(\tth^\ddag_{\La},\tth^\ddag_{\Y}).\]  Thus, as~$\GE\in I_{\tG}(\tth_{\La},\tth_{\Y})$ by Theorem ~\ref{theorem22} \ref{theorem222}, we have
\[\tGE^\ddag\in I_{\tG}(\tth^\ddag_{\La},\tth^\ddag_{\Y}),\]
hence~$\tth^\ddag_{\Y}=\t^\dag(\tth^\ddag_{\La})$, again by Theorem~\ref{theorem22} \ref{theorem222}. Moreover, taking~$x\in\tGE^\ddag$ such
that~$\La^\dag=x\cdot\Y^\dag$, as in Lemma~\ref{lem:conj}, we
have\[\tGE^\ddag\subseteq I_{\tG}(\tth^\ddag_{\La},({\tth}^\ddag_{\La})^x),\]
as~$\tGE^\ddag$ intertwines~$\tth^\ddag_{\La}$ by Theorem \ref{theorem22} \ref{theorem221}.  Since~$({\tth}^\ddag_{\La})^x\in\Cc(\Y,0,\b)$, we deduce that~$\tth^\ddag_{\Y}=({\tth}^\ddag_{\La})^x$ by the unicity
of the transfer in Theorem~\ref{theorem22} \ref{theorem222}.

By Theorem~\ref{theorem22} \ref{theorem221}, we 
have
\[I_{\tG}({\tth}^\ddag_{\La},{\tth}^\ddag_{\La})=(\tJ^{\ddag}_{\La})^1\tGE^{\ddag}(\tJ^{\ddag}_{\La})^1.\]
If~$y\in\tG^\ddag$ then~$y\in
I_{\tG}({\tth}^\ddag_{\La},{\tth}^\ddag_{\La})$ if and only if~$xy\in
I_{\tG}({\tth}^\ddag_{\La},({\tth}^\ddag_{\La})^x)$.
Therefore
\[
I_{\tG}({\tth}^\ddag_{\La},{\tth}^\ddag_{\Y})
=x^{-1}I_{\tG}({\tth}^\ddag_{\La},({\tth}^\ddag_{\La})^x)
=x^{-1}(\tJ_{\La}^{\ddag})^1\tGE^\ddag(\tJ^{\ddag}_{\La})^1 
=(\tJ^{\ddag}_{\Y})^1\tGE^\ddag(\tJ^{\ddag}_{\La})^1.
\]
Now, as in the proof of Lemma \ref{GLssintprop}, let~$\Ga$ be the group~$2$-subgroup of~$\tM^\ddag$ generated by blocks consisting of~$\Id$ and~$-\Id$.  Because~$\Ga$ is a~$2$-group and~$(\tJ^{\ddag}_{\La})^1$
and~$(\tJ^{\ddag}_{\Y})^1$ are pro-$p$ groups, with~$p$ odd, the
non-abelian cohomology pointed
set~$H^1(\Ga,g(\tJ^{\ddag}_{\Y})^1g^{-1}\cap (\tJ^{\ddag}_{\La})^1)$
is trivial, for
all~$g\in\tG$. Hence, by Theorem~\ref{fixedpt1},
\begin{align*}
((\tJ_{\Y}^{\ddag})^1 \tGE^\ddag(\tJ_{\La}^{\ddag})^1)\cap \tM^\ddag
&=((\tJ_{\Y}^{\ddag})^1 \cap \tM^\ddag)(\tGE^\ddag\cap \tM^\ddag)((\tJ_{\La}^{\ddag})^1\cap \tM^\ddag)\\
&=(\tJ_{\Y}^\ddag \cap \tM^\ddag)(\tGE^\ddag\cap \tM^\ddag)(\tJ_{\La}^\ddag\cap \tM^\ddag).
\end{align*}
Finally, for~$g^\dag\in I_{\tG}(\tth^\ddag_{\La},\tth^\ddag_{\Y})$, we have an Iwahori decomposition
\begin{align*}
\tH^1(\b^\dag,\La^\dag)\cap \tH^1(\b^\dag,\Y^\dag)^{g^\dag}&=(\tH^1(\b^\dag,\La^\dag)\cap \tH^1(\b^\dag,\Y^\dag)^{g^\dag}\cap  \widetilde{U}^{-,\dag})\\
&(\tH^1(\b^\dag,\La^\dag)\cap \tH^1(\b^\dag,\Y^\dag)^{g^\dag}\cap  \tM^\ddag)(\tH^1(\b^\dag,\La^\dag)\cap \tH^1(\b^\dag,\Y^\dag)^{g^\dag}\cap  \widetilde{U}^{-,\dag}),
\end{align*}
 and, by Lemma \ref{descthetadag},~${\tth}^\ddag_{\La},{\tth}^\ddag_{\Y}$ are trivial on the unipotent parts of this decomposition. Hence, we have
\[I_{\tM^\ddag}({\tth}^\ddag_{\La}\mid_{\tM^\ddag},{\tth}^\ddag_{\Y}\mid_{\tM^\ddag})=I_{\tG}({\tth}^\ddag_{\La},{\tth}^\ddag_{\Y})\cap \tM^\ddag.\] 
Therefore
\[
I_{\tM^\ddag}({\tth}^\ddag_{\La}\mid_{\tM^\ddag},{\tth}^\ddag_{\Y}\mid_{\tM^\ddag})
=(\tJ_{\Y}^\ddag \cap \tM^\ddag)(\tGE^\ddag\cap \tM^\ddag)(\tJ_{\La}^\ddag\cap \tM^\ddag).
\]
Restricting this equality to a single block in~$\tM^\ddag$ we
recover~$I_{\tG}(\tth_{\La},\tth_{\Y})=\tJ_{\Y}\tGE \tJ_{\La}$.
\end{proof}

Suppose further that~$[\La,n_{\La},0,\b]$ and~$[\Y,n_{\Y},0,\b]$ are
self-dual.  Let~$\th_{\La}\in\Cc_{-}(\La,0,\b)$
and~$\th_{\Y}=\t_{\La,\Y,\b}(\th_{\La})$. Let~$J_{\La}=J(\b,\La)$
and~$J_{\Y}=J(\b,\Y)$.

\begin{thm}\label{classthetaint}  We
  have~$I_G(\th_{\La},\th_{\Y})=J_{\Y}\GE J_{\La}$.
\end{thm}

\begin{proof}
Let~$\tth_{\La}\in\Cc(\La,0,\b)$ and~$\tth_{\Y}\in\Cc(\Y,0,\b)$ be
self-dual semisimple characters which restrict to~$\th_\La$
and~$\th_\Y$ respectively. Since~$\tth_{\Y}$ is the
unique~$\Sigma$-fixed semisimple character restricting to~$\th_\Y$, we
have~$\tth_{\Y}=\t(\tth_{\La})$.  Furthermore, letting~$\mathbf{g}$
denote the Glauberman correspondence (\cf~\cite[\S2]{St00} and
the references therein),~$\th_{\La}=\mathbf{g}(\tth_{\La})$
and~$\th_{\Y}=\mathbf{g}(\tth_\Y)$.
By~\cite[Corollary~$2.5$]{St00},~$I_g(\tth_{\La},\tth_{\Y})\neq 0$ if and only if~$I_g(\mathbf{g}(\tth_{\La}),\mathbf{g}(\tth_\Y))\neq
0$. Therefore,
\[
I_G(\th_{\La},\th_{\Y})=I_G(\tth_{\La},\tth_{\Y})\cap G.
\]
Furthermore,~$I_G(\tth_{\La},\tth_{\Y})=\tJ_{\Y}\tGE \tJ_{\La}$ by Theorem~\ref{intchars1}, and
$
(\tJ_{\Y}\tGE \tJ_{\La})\cap G=(\tJ^1_{\Y}\tGE \tJ^1_{\La})\cap G
=J^1_{\Y}\GE J^1_{\La}=J_{\Y}\GE J_{\La}
$
by Theorem~\ref{fixedpt1} and the semisimple intersection property
Corollary~\ref{ssintprop}.
\end{proof}

\subsection{Some exact sequences}

Let~$[\La,n_{\La},0,\b]$ be a semisimple stratum in~$A$. 
We denote by~$a_{\b}$ the adjoint map given by~$a_{\b}(x)=\b
x-x\b$ for~$x\in A$, and by~$s$ a \emph{tame corestriction} on~$A$
relative to~$F[\b]/F$ (cf. \cite[1.3]{BK93} and \cite[Proposition 3.31]{St05}).

\begin{lemma}\label{exactsequences1}
\begin{enumerate}
\item\label{exactsequences1:1}
Let~$[\La,n_{\La},0,\b]$ be a semisimple stratum in~$A$.  The sequence 
\begin{center}
\begin{tikzpicture}[>=stealth]
\matrix (m) [matrix of math nodes, column sep=2.5em, row sep=0.5em,text height=1.5ex, text depth=0.25ex]
{
 {0}&{\QQ(\b,\La)}& {\JJ^1(\b,\La)} & {{\HH^1(\b,\La)}^*}&{\BB(\b,\La)}&0\\
};
\path[overlay,->, font=\scriptsize,>=latex]
        (m-1-1) edge (m-1-2)
        (m-1-2) edge (m-1-3)
        (m-1-3) edge node[above] {$a_{\b}$} (m-1-4)
        (m-1-4) edge node[above] {$s$}  (m-1-5)
        (m-1-5) edge (m-1-6);
\end{tikzpicture}
\end{center}
is exact.
\item\label{exactsequences1:2} 
Let~$[\La,n_{\La},0,\b]$ and~$[\Y,n_{\Y},0,\b]$ be semisimple strata in~$A$ and~$y\in\tGE$.  The sequence 
\begin{center}
\begin{tikzpicture}[>=stealth,descr/.style={fill=white,inner sep=1.5pt}]
\matrix (m) [matrix of math nodes, column sep=2.5em, row sep=1.5em,text height=1.5ex, text depth=0.25ex, ]
{
 {0}& {\QQ(\b,\La)+(\QQ(\b,\Y))^y}& {\JJ^1(\b,\La)+(\JJ^1(\b,\Y))^y}\\
& {\HH^1(\b,\La)^*+(\HH^1(\b,\Y)^*)^y}&{\BB(\b,\La)+(\BB(\b,\Y))^y}& {0}\\
};
       \path[overlay,->, font=\scriptsize,>=latex]
        (m-1-1) edge (m-1-2)
        (m-1-2) edge (m-1-3)
        (m-1-3) edge[out=355,in=175]  node[descr] {$\a_{\b}$} (m-2-2)
        (m-2-2) edge node[above] {$s$} (m-2-3)
        (m-2-3) edge (m-2-4);
\end{tikzpicture}
\end{center}
is exact.
\end{enumerate}
\end{lemma}

\begin{proof}
When~$\La=\Y$, both parts follow from~\cite[Lemma~3.17]{St05}
(\cf~[\emph{op.\,cit.},~Proposition~3.31]). Passing to~$\dag$ we have
the second exact sequence for the semisimple
strata~$[\La^\dag,n_{\La},0,\b^\dag]$
and~$[\Y^\dag,n_{\La}^\dag,0,\b]$, by choosing~$x\in\tGE^\ddag$ as in
Lemma~\ref{lem:conj}, and replacing~$y$ by~$xy$ in the exact sequence
for~$\La^\dag$. Intersecting with a single block
we have~\ref{exactsequences1:2}, while~\ref{exactsequences1:1} is the
special case~$y=1$.
\end{proof}

When we have a self-dual semisimple stratum~$[\La,n_{\La},0,\b]$, we
may (and do) choose a tame corestriction~$s$ which commutes with the
anti-involution~$\s$ on~$A$  (cf. \cite[2.1.1]{St98}).  Then we get the self-dual analogue of Lemma~\ref{exactsequences1}.  
 
\begin{lemma}\label{exactsequences2}
\begin{enumerate}
\item\label{exactsequences2:1} 
Let~$[\La,n_{\La},0,\b]$ be a self-dual semisimple stratum
in~$A$.  The sequence
\begin{center}
\begin{tikzpicture}[>=stealth]
\matrix (m) [matrix of math nodes, column sep=2.5em, row sep=0.5em,text height=1.5ex, text depth=0.25ex]
{
 {0}&{\QQ_{\La}^{-}(\b,\La)}& {\JJ^1_{-}(\b,\La)} & {{\HH^1_{-}(\b,\La)}^*}&{\BB^{-}(\b,\La)}&0\\
};
\path[overlay,->, font=\scriptsize,>=latex]
        (m-1-1) edge (m-1-2)
        (m-1-2) edge (m-1-3)
        (m-1-3) edge node[above] {$a_{\b}$} (m-1-4)
        (m-1-4) edge node[above] {$s$}  (m-1-5)
        (m-1-5) edge (m-1-6);
\end{tikzpicture}
\end{center}
is exact.
\item\label{exactsequences2:2} 
Let~$[\La,n_{\La},0,\b]$ and~$[\Y,n_{\Y},0,\b]$ be self-dual
semisimple strata in~$A$ and~$y\in\GE^+$.  The sequence
\begin{center}
\begin{tikzpicture}[>=stealth,descr/.style={fill=white,inner sep=1.5pt}]
\matrix (m) [matrix of math nodes, column sep=2.5em, row sep=1.5em,text height=1.5ex, text depth=0.25ex, ]
{
 {0}& {\QQ^{-}(\b,\La)+(\QQ^{-}(\b,\Y))^y}& {\JJ^1_{-}(\b,\La)+(\JJ^1_{-}(\b,\Y))^y}\\
& {\HH^1_{-}(\b,\La)^*+(\HH^1_{-}(\b,\Y)^*)^y}&{\BB^{-}(\b,\La)+(\BB^{-}(\b,\Y))^y}& {0}\\
};
       \path[overlay,->, font=\scriptsize,>=latex]
        (m-1-1) edge (m-1-2)
        (m-1-2) edge (m-1-3)
        (m-1-3) edge[out=355,in=175]  node[descr] {$\a_{\b}$} (m-2-2)
        (m-2-2) edge node[above] {$s$} (m-2-3)
        (m-2-3) edge (m-2-4);
\end{tikzpicture}
\end{center}
is exact.
\end{enumerate}
\end{lemma}

\section{Intertwining of Heisenberg representations}

While up to now, we have been generalising results for both~$\tG$
and~$G$ in this section we concern ourself only with representations
of~$G$. The same methods apply for representations of~$\tG$.

Let~$[\La,n_{\La},0,\b]$ and~$[\Y,n_{\Y},0,\b]$ be self-dual
semisimple strata in~$A$. In this section we will abbreviate lattices
in~$A^-$ without the superscript~${}^-$, to simplify the
notation. Thus we
write~$\QQ_{\La}=\QQ^{-}(\b,\La)$,~$\HH_{\La}=\HH^1_-(\b,\La)$,~$\JJ_{\La}=\JJ^1_-(\b,\La)$,
and~$\BB_{\La}=\BB^{-}(\b,\La)$, using analogous notation
for~$\Y$. (Note, in particular, that we are omitting the
superscript~${}^1$ here.) We
also write~$H^1_{\La}=H^1(\b,\La)$ and~$J^1_{\La}=J^1(\b,\La)$,
with~$H^1_\Y,J^1_\Y$ defined similarly.

Let~$\th_{\La}\in\Cc_{-}(\La,0,\b)$
and~$\th_{\Y}=\t_{\La,\Y,\b}(\th_{\La})$.  Let~$\eta_{\La}$ be the
unique Heisenberg representation containing~$\th_{\La}$
and~$\eta_{\Y}=\t_{\La,\Y,\b}(\eta_{\La})$ the unique Heisenberg
representation containing~$\th_{\Y}$.

\begin{thm}\label{intetas} 
The intertwining of~$\eta_{\La}$ and~$\eta_{\Y}$ in~$G$ is given by
\[
\dim_{R}(I_g(\eta_\La,\eta_\Y))=
\begin{cases} 
1&\text{if }g\in J_\Y\GE J_\La;\\ 
0&\text{otherwise.}
\end{cases}
\]
\end{thm}

This theorem is an asymmetric generalisation
of~\cite[Proposition~5.1.8]{BK93} in the classical groups setting (see
also~\cite[Proposition~3.31]{St05}) and we imitate those proofs. 

\begin{lemma}\label{doublecosets1}
For any~$y\in\GE^+$, we have
\[
(J_\La^1: J_\La^1\cap yJ_\Y^1 y^{-1})(J_\Y^1: y^{-1}J_\La^1 y\cap J_\Y^1)
=
(H_\La^1: H_\La^1\cap yH_\Y^1 y^{-1})(H_\Y^1: y^{-1}H_\La^1 y\cap H_\Y^1).
\]
\end{lemma}

\begin{proof}
We begin by recalling the following from~\cite{BK93}: 
let~$0\to V_1\to V_2\to V_3\to V_4\to 0$ be an exact
sequence of finite-dimensional~$F$-vector spaces and, for~$1\le i\le
4$, let~$\mu_i$ be an~$F$-Haar measure on~$V_i$. 
By~\cite[Lemma~5.1.3]{BK93}, there is a constant~$c\in F^\times$ such
that, if the sequence restricts to an exact sequence~$0\to L_1\to
L_2\to L_3\to L_4\to 0$ of~$\o_F$-lattices~$L_i$ in~$V_i$, then
\[
\frac{\mu_1(L_1)\mu_3(L_3)}{\mu_2(L_2)\mu_4(L_4)}=c.
\]
Moreover,~$\mu_1(L_1)\mu_1(L_1^*)$ is also independent of 
the~$\o_F$-lattice~$L_1$, by~\cite[Lemma~5.1.5]{BK93}.

We have such an exact sequence
\begin{center}
\begin{tikzpicture}[>=stealth]
\matrix (m) [matrix of math nodes, column sep=2.5em, row sep=0.5em,text height=1.5ex, text depth=0.25ex]
{
 {0}&{B}& {A} & {A}&{B}&0,\\
};
\path[overlay,->, font=\scriptsize,>=latex]
        (m-1-1) edge (m-1-2)
        (m-1-2) edge node[above] {$s$} (m-1-3)
        (m-1-3) edge node[above] {$a_{\b}$} (m-1-4)
        (m-1-4) edge node[above] {$s$}  (m-1-5)
        (m-1-5) edge (m-1-6);
\end{tikzpicture}
\end{center}
and, choosing~$F$-Haar measures~$\mu_A$ on~$A$ and~$\mu_B$ on~$B$, we
denote by~$c\in F^\times$ the invariant given
by~\cite[Lemma~5.1.3]{BK93}, as above. Now we
apply this to the rows of the following giant commutative
diagram of~$\o_F$-lattices, which we get from
Lemma~\ref{exactsequences2}\ref{exactsequences2:2}.
\begin{center}
\begin{tikzpicture}[>=stealth]
\matrix [matrix of math nodes, column sep=1cm, row sep=1cm,text height=1.5ex, text depth=0.25ex]
{
\node(row11) {}; &\node(row12) {0};&\node(row13) {0}; &\node(row14) {0};&\node(row15) {0}; &\node(row16) {};\\
\node(row21) {0}; &\node(row22) {\QQ_\La\cap\QQ_\Y^y};&\node(row23) {\JJ_\La\cap\JJ_\Y^y}; &\node(row24) {\HH_\La^*\cap(\HH_\Y^*)^y};&\node(row25) {\BB_\La\cap\BB_\Y^y}; &\node(row26) {0};\\
\node(row31) {0}; &\node(row32) {\QQ_\La\oplus\QQ_\Y^y};&\node(row33) {\JJ_\La\oplus\JJ_\Y^y}; &\node(row34) {\HH_\La^*\oplus(\HH_\Y^*)^y};&\node(row35) {\BB_\La\oplus\BB_\Y^y}; &\node(row36) {0};\\
\node(row41) {0}; &\node(row42) {\QQ_\La+\QQ_\Y^y};&\node(row43) {\JJ_\La+\JJ_\Y^y}; &\node(row44) {\HH_\La^*+(\HH_\Y^*)^y};&\node(row45) {\BB_\La+\BB_\Y^y}; &\node(row46) {0};\\
\node(row51) {}; &\node(row52) {0};&\node(row53) {0}; &\node(row54) {0};&\node(row55) {0}; &\node(row56) {};\\
};
\draw [->] (row21) -- (row22);
\draw [->] (row22) -- (row23);
\draw [->] (row23) -- (row24);
\draw [->] (row24) -- (row25);
\draw [->] (row25) -- (row26);

\draw [->] (row31) -- (row32);
\draw [->] (row32) -- (row33);
\draw [->] (row33) -- (row34);
\draw [->] (row34) -- (row35);
\draw [->] (row35) -- (row36);

\draw [->] (row41) -- (row42);
\draw [->] (row42) -- (row43);
\draw [->] (row43) -- (row44);
\draw [->] (row44) -- (row45);
\draw [->] (row45) -- (row46);

\draw[->] (row12) -- (row22);
\draw[->] (row13) -- (row23);
\draw[->] (row14) -- (row24);
\draw[->] (row15) -- (row25);

\draw[->] (row22) -- (row32);
\draw[->] (row23) -- (row33);
\draw[->] (row24) -- (row34);
\draw[->] (row25) -- (row35);

\draw[->] (row32) -- (row42);
\draw[->] (row33) -- (row43);
\draw[->] (row34) -- (row44);
\draw[->] (row35) -- (row45);

\draw[->] (row42) -- (row52);
\draw[->] (row43) -- (row53);
\draw[->] (row44) -- (row54);
\draw[->] (row45) -- (row55);
\end{tikzpicture}
\end{center}
Using the first row, we get
\[
\frac{\mu_A(\HH_{\La}^*\cap(\HH_{\Y}^*)^y)}{\mu_A(\JJ_{\La}\cap\JJ_{\Y}^y)}=c\,\frac{\mu_B(\BB_{\La}\cap\BB_{\Y}^y)}{\mu_B(\QQ_{\La}\cap\QQ_{\Y}^y)}.
\]
we have~$\mu_B(\QQ_{\La}\cap\QQ_{\Y}^y)=\mu_B(\QQ_{\La})\mu_B(\QQ_{\Y})/\mu_B(\QQ_{\La}+\QQ_{\Y}^y)$, 
from the first column, and similarly
for~$\mu_A(\HH^*_\La\cap(\HH^*_\Y)^y)$, whence
\[
\frac{\mu_A(\HH^*_\La)\mu_A(\HH^*_{\Y})}{\mu_A(\JJ_{\La}\cap\JJ_{\Y}^y)\mu_A(\HH_{\La}^*+(\HH_{\Y}^*)^y)}=c\,\frac{\mu_B(\BB_{\La}\cap\BB_{\Y}^y)\mu_B(\QQ_{\La}+\QQ_{\Y}^y)}{\mu_B(\QQ_{\La})\mu_B(\QQ_{\Y})},
\]
Since~$(\HH_{\La}\cap\HH_{\Y}^y)^*=\HH_{\La}^*+(\HH_{\Y}^*)^y$,
from~\cite[Lemma~5.1.5]{BK93} we have
\[
\mu_A(\HH_{\La}^*+(\HH_{\Y}^*)^y)\mu_A(\HH_{\La}\cap\HH_{\Y}^y)=
(\mu_A(\HH_{\La})\mu_A(\HH_{\La}^*)\mu_A(\HH_{\Y})\mu_A(\HH_{\Y}^*))^{\frac{1}{2}},
\]
with a similar result
using~$(\BB_{\La}\cap\BB_{\Y}^y)^*=\QQ_{\La}+\QQ_{\Y}^y$. Substituting,
we get
\[
\frac{\mu_A(\HH_{\La}\cap\HH_{\Y}^y)}{\mu_A(\JJ_{\La}\cap\JJ_{\Y}^y)} 
\left(\frac{\mu_A(\HH^*_\La)\mu_A(\HH^*_\Y)}{\mu_A(\HH_\La)\mu_A(\HH_\Y)}\right)^{\frac 12}
= c \left(\frac{\mu_B(\BB_\La)\mu_B(\BB_\Y)}{\mu_B(\QQ_\La)\mu_B(\QQ_\Y)}\right)^{\frac 12}.
\]
Finally, from Lemma~\ref{exactsequences1}\ref{exactsequences1:1}, we
have
\[
\frac{\mu_A(\HH^*_\La)}{\mu_A(\JJ_{\La})}=c\,\frac{\mu_B(\BB_{\La})}{\mu_B(\QQ_{\La})},
\]
and similarly for~$\Y$, which gives
\[
\frac{\mu_A(\HH_{\La}\cap\HH_{\Y}^y)}{\mu_A(\JJ_{\La}\cap\JJ_{\Y}^y)}=
\left(\frac{\mu_A(\HH_{\La})\mu_A(\HH_{\Y})}{\mu_A(\JJ_{\La})\mu_A(\JJ_{\Y})}\right)^{\frac 12}.
\]
Conjugating by~$y$, we get the same formula
for~$\mu_A(\presuper{y}{\HH_{\La}}\cap\HH_{\Y})/\mu_A(\presuper{y}{\JJ_{\La}}\cap\JJ_{\Y})$. Multiplying
these and rearranging, we get
\[
\(\frac{\mu_A(\HH_{\La})}{\mu_A(\HH_{\La}\cap\HH_{\Y}^y)}\)\(\frac{\mu_A(\HH_{\Y})}{\mu_A(\presuper{y}{\HH_{\La}}\cap\HH_{\Y})}\)=\(\frac{\mu_A(\JJ_{\La})}{\mu_A(\JJ_{\La}\cap\JJ_{\Y}^y)}\)\(\frac{\mu_A(\JJ_{\Y})}{\mu_A(\presuper{y}{\JJ_{\La}}\cap\JJ_{\Y})}\).
\]
The result follows from this additive statement since~$H^1_\La$ is the
image under the Cayley transform of~$\HH_\La$, and similarly for the
other groups involved.
\end{proof}

\begin{lemma}\label{counting}
For any~$y\in\GE^+$, we have
\[
\left|H^1_\Y\backslash J_{\Y}^1 yJ_{\La}^1/H_{\La}^1\right|=
(J_{\La}^1:H_{\La}^1)^{\frac 12}(J_{\Y}^1:H_{\Y}^1)^{\frac 12}.
\]
\end{lemma}

\begin{proof}
Fix a~$F$-Haar measure~$\mu$ on~$G$. Decomposing~$J^1_\Y yJ^1_\La$ by
right~$J^1_\Y$-cosets, and by left~$J^1_\La$-cosets, and then
multiplying, we have 
\[
\mu(J^1_{\Y}yJ^1_{\La})^2=\mu(J_{\La}^1)\mu(J_{\Y}^1)(J_{\La}^1: J_{\La}^1\cap y^{-1}J_{\Y}^1 y)(J_{\Y}^1: yJ_{\La}^1y^{-1}\cap J_{\Y}^1).
\]
By normality of~$H_\La^1$ in~$J_\La^1$ and~$H_\Y^1$ in~$J_\Y^1$, for any~$y'\in
J_{\Y}^1yJ_{\La}^1$ we similarly have
\[
\mu(H^1_{\Y}y'H^1_{\La})^2=\mu(H_{\La}^1)\mu(H_{\Y}^1)(H_{\La}^1: H_{\La}^1\cap y^{-1}H_{\Y}^1 y)(H_{\Y}^1: yH_{\La}^1y^{-1}\cap H_{\Y}^1).
\]
Therefore, we have~$\left|H^1_{\La}\backslash J_{\La}^1
  gJ_{\Y}^1/H_{\Y}^1\right|=(J_{\La}^1:H_{\La}^1)^{\frac 12}
(J_{\Y}^1:H_{\Y}^1)^{\frac 12}$, by Lemma~\ref{doublecosets1}.
\end{proof}

\begin{proof}[Proof of Theorem {\ref{intetas}}]
By~\cite[Corollary~3.29]{St05} the induced
representation~$\ind_{H_{\La}^1}^{J_{\La}^1}(\th_{\La})$ is a multiple
of~$\eta_\La$, that multiple being~$(J_{\La}^1:H_{\La}^1)^{\frac 12}$,
and analogously for~$\th_{\Y}$. Thus   
\[
\dim_{R}(I_g(\ind_{H_{\La}^1}^{J_{\La}^1}\th_{\La},\ind_{H_{\Y}^1}^{J_{\Y}^1}\th_{\Y}))=
(J_{\La}^1:H_{\La}^1)^{\frac 12}(J_{\Y}^1:H_{\Y}^1)^{\frac 12}\dim_{R}(I_g(\eta_{\La},\eta_{\Y})).
\]
By Lemma~\ref{heckelemma}\ref{heckelemma:3},
\[
\Hh(G,\ind_{H_{\La}^1}^{J_{\La}^1}\th_{\La},\ind_{H_{\Y}^1}^{J_{\Y}^1}\th_{\Y})_g\simeq \coprod_{\substack{h\in H_{\La}^1\backslash G/H_{\Y}^1\\J_{\La}^1hJ_{\Y}^1=J_{\La}^1gJ_{\Y}^1}}\Hh(G,\th_{\La},\th_{\Y})_h.
\]
Therefore, by Theorem~\ref{classthetaint} and Lemma~\ref{counting}, we have
\[
\dim_{R}(I_g(\ind_{H_{\La}^1}^{J_{\La}^1}\th_{\La},\ind_{H_{\Y}^1}^{J_{\Y}^1}\th_{\Y}))=\begin{cases}(J_{\La}^1:H_{\La}^1)^{\frac
    12}(J_{\Y}^1:H_{\Y}^1)^{\frac 12} &\text{if }g\in J_{\Y}\GE J_{\La};\\
0&\text{otherwise,}
\end{cases}
\]
whence the result.
\end{proof}

\begin{remark}
In the setting of Theorem~\ref{classthetaint}, we also
have~$I_{G^+}(\th_{\La},\th_{\Y})=J_{\Y}^1\GE^+ J_{\La}^1$ by
intersecting the intertwining of~$I_{\tG}(\tth_{\La},\tth_{\Y})$
with~$G^+$ rather than~$G$.  Moreover, in the setting of
Theorem~\ref{intetas} the same proof shows that the intertwining
of~$\eta_{\La}$ and~$\eta_{\Y}$ in~$G^+$ is given by
\[ 
\dim_{R}(I_g(\eta_\La,\eta_\Y))=
\begin{cases} 1&\text{if }g\in J_\Y^1\GE^+ J_\La^1;\\ 0&\text{otherwise.}\end{cases}
\]
\end{remark}

We will also make use of the following lemma of~\cite{St08}.

\begin{lemma}[{\cite[Lemma 3.6]{St08}}]\label{dimetascom}
We have~${\dim(\eta_{\La})}/{\dim(\eta_{\Y})}={(J_{\La}^1:J_{\Y}^1)}/{(P^1(\La_E):P^1(\Y_E))}$.
\end{lemma}

Conjugating if necessary, we assume that~$\BB(\La)$ and~$\BB(\Y)$
contain a common minimal self-dual hereditary order~$\BB(\Ga)$
corresponding to an~$\o_E$-lattice sequence~$\Ga$ in~$V$;
thus~$P^\circ(\Ga_E)$ is an Iwahori subgroup
of~$\GE$. Let~$\th_{\Ga}=\t_{\La,\Ga,\b}(\th_{\La})=\t_{\Y,\Ga,\b}(\th_{\Y})\in\Cc_{-}(\Ga,0,\b)$.
Let~$\eta_{\Ga}$ be the unique Heisenberg representation
containing~$\th_{\Ga}$ and let~$J_{\Ga}=J(\b,\Ga)$.
Since~$P^1(\Ga_E)$ normalises~$J^1_{\La}$ and~$J^1_{\Y}$ we can form
the groups~$J^1_{\Ga,\La}=P^1(\Ga_E)J^1_{\La}$
and~$J^1_{\Ga,\Y}=P^1(\Ga_E)J^1_{\Y}$.

\begin{lemma}[{\cite[Proposition 3.7]{St08}}]\label{compatetas}
There exist unique irreducible representations~$\eta_{\Ga,\La}$ of~$J^1_{\Ga,\La}$ and~$\eta_{\Ga,\Y}$ of~$J^1_{\Ga,\Y}$ such that
\begin{enumerate}
\item~$\eta_{\Ga,\La}\mid_{J^1_{\La}}=\eta_{\La}$
  and~$\eta_{\Ga,\Y}\mid_{J^1_{\Y}}=\eta_{\Y}$; 
\item~$\eta_{\Ga,\La}$,~$\eta_{\Ga,\Y}$ and~$\eta_{\Ga}$ induce
  equivalent irreducible representations of~$P^1(\Ga)$. 
\end{enumerate}
\end{lemma}

We can now extend the intertwining result of~\cite[Proposition
3.7]{St08}.  The proof is essentially the same as that of~\cite[Proposition
5.1.19]{BK93}. 

\begin{lemma}\label{intetas2}
The intertwining of~$\eta_{\Ga,\La}$ and~$\eta_{\Ga,\Y}$ in~$G$ is given by
\[
\dim_R (I_g(\eta_{\Ga,\La},\eta_{\Ga,\Y}))=
\begin{cases} 1&\text{if }g\in J^1_{\Ga,\Y}\GE J^1_{\Ga,\La};\\ 0&\text{otherwise.}\end{cases}
\]
\end{lemma}

We remark that~$J^1_{\Ga,\Y}\GE J^1_{\Ga,\La}=J^1_{\Y}\GE J^1_{\La}$,
and that we have a similar result for the intertwining in~$G^+$.

\begin{proof}
We have~$I_G(\eta_{\Ga,\La},\eta_{\Ga,\Y})\subseteq
I_G(\eta_\La,\eta_\Y)=J^1_{\Y}\GE J^1_{\La}$ and the non-zero
intertwining spaces are one-dimensional by
Lemma~\ref{intetas}. If~$x\in \GE$ then~$x\in I_G(\eta_{\Ga})$, by
Theorem~\ref{intetas}, so~$x\in
I_G(\Ind_{J^1_{\Ga}}^{P^1(\Ga)}(\eta_{\Ga}))$.  Thus~$x\in
I_G(\Ind_{J^1_{\Ga,\La}}^{P^1(\Ga)}(\eta_{\Ga,\La}),
\Ind_{J^1_{\Ga,\Y}}^{P^1(\Ga)}(\eta_{\Ga,\Y}))$ by
Lemma~\ref{compatetas}.  Therefore there exist~$u,v\in P^1(\Ga)$ such
that~$uxv\in I_G(\eta_{\Ga,\La}, \eta_{\Ga,\Y})$; since this
intertwining set is contained in~$J^1_{\Y}\GE J^1_{\La}$, there exist~$j_{\La}\in J^1_{\La}$ and~$j_{\Y}\in J^1_{\Y}$ such that~$j_{\Y}uxvj_{\La}\in \GE$.  By Corollary~\ref{ssintprop},~$P^1(\Ga)xP^1(\La)\cap \GE=P^1(\Ga_E)xP^1(\La_E)$.  
Therefore, we can find~$u'\in P^1(\Ga_E)$ and~$v'\in P^1(\La_E)$ such
that~$u'xv'=j_1uxvj_2$, whence~$x\in I_G(\eta_{\Ga,\La}, \eta_{\Ga,\Y})$.
\end{proof}

\section{$\b$-extensions}\label{betasects}

We generalise the definition of~$\b$-extensions for classical groups,
as defined by the second author when~$R=\CC$ in~\cite{St08}.  As
the~$J$ groups are not pro-$p$, the proofs of the corresponding
statements need to be adapted in characteristic~$\ell$.  However, as
the~$J^1$ groups are pro-$p$, these modifications are relatively
simple.

Let~$[\La,n_{\La},0,\b]$ be a self-dual semisimple
stratum,~$\th_{\La}\in\Cc_{-}(\La,0,\b)$ and~$\eta_{\La}$ the unique
Heisenberg representation containing~$\th_{\La}$. We will
write~$\BB(\La_E)=\BB(\b,\La)$ for the hereditary~$\o_E$-order in~$B$
determined by the lattice sequence~$\La$, and will
abbreviate~$J_\La^+=J^+(\b,\La)$, etc.

\begin{thm}\label{maxbetas}
Let~$\Ga$ be any self-dual~$\o_E$-lattice sequence such
that~$\BB(\Ga_E)$ is a minimal self-dual~$\o_E$-order in~$B$ contained
in~$\BB(\La_E)$.  There exists a representation~$\k_{\La}^+$
of~$J_{\La}^+$ extending~$\eta_{\Ga,\La}$.  Moreover, any two such
extensions differ by a character of~$P^+(\La_E)/P^1(\La_E)$ which is
trivial on the subgroup generated by all its unipotent subgroups.
\end{thm}

\begin{proof}
The proof follows \emph{mutatis mutandis} the proof of~\cite[Theorem~4.1]{St08}.
\end{proof}

If~$\BB(\La_E)$ is a maximal self-dual~$\o_E$-order in~$B$, we call an
extension~$\k_{\La}^+$ of~$\eta_{\La}$, as constructed in
Theorem~\ref{maxbetas}, a \emph{$\b$-extension}.  In the case
where~$\BB(\La_E)$ is not maximal, while Theorem~\ref{maxbetas}, gives
a collection of extensions of~$\eta_{\La}$ it gives too many such
extensions.  As in the complex case, we define~$\b$-extensions in the
non-maximal case by compatibility with~$\b$-extensions in the maximal
case.   Let~$[\Y,n_{\Y},0,\b]$ be a self-dual semisimple stratum such
that~$\BB(\Y_E)$ is maximal and~$\BB(\La_E)\subseteq \BB(\Y_E)$;
let~$\th_{\Y}=\t_{\La,\Y,\b}(\th_{\La})$
and~$\eta_{\Y}=\t_{\La,\Y,\b}(\eta_{\La})$.
Let~$J^1_{\La,\Y}=P^1(\La_E)J^1_\Y$
and~$J^+_{\La,\Y}=P^+(\La_E)J^1_\Y$.

\begin{thm}\label{bijectionbetas}
There is a canonical bijection 
\[
b_{\La,\Y}:\{\text{extensions~$\k_{\La}^+$ of~$\eta_{\La}$ to~$J^+_\La$}\}\to \{\text{extensions~$\k^+_{\La,\Y}$ of~$\eta_{\Y}$ to~$J^+_{\La,\Y}$}\}.
\]
Furthermore, if~$\AA(\La)\subseteq \AA(\Y)$
then~$b_{\La,\Y}(\k^+_{\La})$ is the unique extension of~$\eta_{\Y}$
such that~$\k^+_\La$ and~$b_{\La,\Y}(\k^+_{\La})$
induce to equivalent irreducible representations of~$P^+(\La_E)P^1(\La)$.
\end{thm}

\begin{proof}
Assume that~$\AA(\La)\subseteq \AA(\Y)$ and, as in the proof
of~\cite[Lemma~4.3, Case (i)]{St08}, we follow the argument
of~\cite[Proposition~5.2.5]{BK93}. 
Let~$\k_\La^+$ be an extension of~$\eta_\La^+$ to~$J^+_\La$ and put
\[
\l=\ind_{J_{\La}^+}^{P^+(\La_E)P^1(\La)}(\k_{\La}^+).
\]
By Mackey Theory,
\[
\Res^{P^+(\La_E)P^1(\La)}_{P^1(\La)}(\l)
\simeq \ind_{J^1_{\La}}^{P^1(\La)}(\eta_{\La}),
\]
which is irreducible, since~$I_G(\eta_\La)\cap P^1(\La)=J^1_\La$; in
particular,~$\l$ is irreducible. Moreover, by Lemma~\ref{compatetas},
\[
\Res^{P^+(\La_E)P^1(\La)}_{P^1(\La)}(\l)
\simeq \ind_{J^1_{\La,\Y}}^{P^1(\La)}(\eta_{\La,\Y}),
\]
so there is an irreducible quotient~$\k_{\La,\Y}^+$
of~$\Res^{P^+(\La_E)P^1(\La)}_{J_{\La,\Y}^+}\l$ which
contains~$\eta_{\La,\Y}$; indeed, there is a unique such quotient,
since~$\eta_{\La,\Y}$ appears with multiplicity~$1$
in~$\Res^{P^1(\La)}_{J^1(\La,\Y)}\ind_{J^1_{\La,\Y}}^{P^1(\La)}(\eta_{\La,\Y})$,
by Lemma~\ref{intetas2}. Now put
\[
\l'=\ind_{J_{\La,\Y}^+}^{P^+(\La_E)P^1(\La)}\k_{\La,\Y}^+.
\]
Then, as above,
\[
\Res^{P^+(\La_E)P^1(\La)}_{P^1(\La)}(\l')
\simeq \ind_{J^1_{\La,\Y}}^{P^1(\La)}(\eta_{\La,\Y}),
\]
so that~$\l'$ is also irreducible, and hence equivalent
to~$\l$. Comparing dimensions, using Lemma~\ref{dimetascom}, we see
that~$\k_{\La,\Y}^+$ extends~$\eta_{\La,\Y}$ as required. 

The argument is reversible, giving the required bijection, and the
remainder of the proof follows from this special
case~$\AA(\La)\subseteq \AA(\Y)$, exactly as in the proof
of~\cite[Lemma~4.3]{St08}.
\end{proof}

An extension~$\k_{\La}^+$ of~$\eta_{\La}$ to~$J^+_{\La}$ is called a
\emph{$\b$-extension} if there exist a self-dual semisimple
stratum~$[\Y,n_{\Y},0,\b]$ such that~$\BB(\Y_E)$ is a maximal
self-dual~$\o_E$-order containing~$\BB(\La_E)$ and
a~$\b$-extension~$\k_{\Y}^+$ of~$\eta_{\Y}=\t_{\La,\Y,\b}(\eta_{\La})$
such that~$b_{\La,\Y}(\k_{\La}^+)=\Res_{J^+_{\La,\Y}}^{J^+_{\Y}}(\k_{\Y}^+)$. 
More precisely, we say that such a representation~$\k_\La^+$ is a
\emph{$\b$-extension relative to~$\Y$.}

There is a standard (non-canonical) choice for the
self-dual~$\o_E$-lattice sequence~$\Y$. Let
\begin{equation*}
\MM_{\La}^i(2r+s)=\begin{cases}
\p_{E_i}^r\La^i(0)&\text{if }i\in I_+;\\
\p_{E_i}^r\La^i(s)&\text{if }i\in I_0;\\
\p_{E_i}^r\La^i(1)&\text{if }i\in I_-.
\end{cases}
\end{equation*}
Then~$\MM_{\La}=\bigoplus_{i\in I}\MM^i_{\La}$ is a
self-dual~$\o_E$-lattice sequence in~$V$ such that~$\AA(\MM_{\La})\cap
B$ is a maximal self-dual hereditary~$\o_E$-order in~$B$.  A
representation~$\k_{\La}^+$ of~$J_{\La}^+$ is called a
\emph{standard~$\b$-extension} of~$\eta_{\La}$ if it is
a~$\b$-extension relative to~$\MM_{\La}$.

If~$\k_\La^+$ is a standard~$\beta$ extension and~$[\Y,n_\Y,0,\b]$ is
another self-dual semisimple stratum with~$\MM_{\Y}=\MM_{\La}$, we say
that the standard~$\b$-extension~$\k_\Y^+$ of~$J_\Y^+$ is
\emph{compatible~$\k_\La^+$} if they correspond to the
same~$\b$-extension of~$J_{\MM_\La}^+$. In the case
that~$\AA(\La)\subseteq\AA(\Y)$, this is equivalent to saying
that~$\k_\La^+$ and~$\Res^{J^+_\Y}_{J^+_{\La,\Y}}\k_\Y^+$ induce to
equivalent (irreducible) representations of~$P^+(\La_E)P^1(\La)$.

 We also call the restriction from~$J^+_{\La}$ to~$J_{\La}$
 (resp.~$J^{\circ}_{\La}$) of a (standard)~$\b$-extension a
 (standard)~$\b$-extension and denote the restriction of~$\k^+_{\La}$
 to~$J_{\La}$ (resp.~$J^{\circ}_{\La}$) by~$\k_{\La}$
 (resp.~$\k^{\circ}_{\La}$), and speak of compatibility for these
 standard~$\b$-extensions.
 
\begin{remark}
Being smooth representations of a compact group, all~$\Ql$-beta
extensions are integral. When~$\BB(\La_E)$ is a maximal
self-dual~$\o_E$-order in~$B$, it is straightforward to check that
reduction modulo-$\ell$ defines a surjective map from the set
of~$\Ql$-beta extensions to the set of~$\Fl$-beta
extensions. Moreover, the bijections~$b_{\La,\Y}$, for~$\Ql$-representation
and~$\Fl$-representations, defined by Theorem~\ref{bijectionbetas}
commute with reduction modulo-$\ell$; thus reduction modulo-$\ell$
defines a surjective map from the set of~$\Ql$-beta extensions to the
set of~$\Fl$-beta extensions in all cases.  Moreover, the reduction
modulo-$\ell$ of a standard~$\Ql$-beta extension is a
standard~$\Fl$-beta extension.
\end{remark}

\subsection{Induction functors for classical groups}\label{Functors}

Now suppose that~$[\La,n,0,\b]$ is a skew semisimple stratum in~$A$.
Let~$\th\in\Cc_{-}(\La,0,\b)$, let~$\eta$ be the unique Heisenberg
extension of~$\th$ to~$J^1(\b,\La)$ and~$\k$ be a~$\b$-extension
of~$\eta$ to~$J(\b,\La)$. Recall that we have an exact sequence
\[
1\to J^1(\b,\La)\to J(\b,\La) \to M(\La_E) \to 1,
\]
with~$M(\La_E)$ a (possibly disconnected) finite reductive group.

We have a functor~$\I_{\k}:\RR_R(M(\La_E))\to \RR_R(G)$, which we
call~$\k$-induction, given by
\[
\I_{\k}(-)=\ind_{J(\b,\La)}^{G}(\k\otimes\infl_{M(\La_E)}^{J(\b,\La)}(-))
\]
where~$\infl_{M(\La_E)}^{J(\b,\La)}:\RR_R(M(\La_E))\to
\RR_R(J(\b,\La))$ is the functor defined by trivial inflation
to~$J^1(\b,\La)$. The functor~$\I_{\k}$ possesses a right
adjoint~$\R_{\k}:\RR_R(G)\to \RR_R(M(\La_E))$, which we
call~$\k$-restriction, given by
\[
\R_{\k}(-)=\Hom_{J^1(\b,\La)}(\k,-).
\]
If~$\pi$ is a smooth representation of~$G$, the action of~$M(\La_E)$
on~$\R_{\k}(\pi)$ is given as follows: if~$f\in\R_{\k}(\pi)$,~$m\in
M(\La_E)$ and~$j\in J(\b,\La)$ is any representative for~$m$,
then~$m\cdot f=\pi(j)\circ f\circ \k(j^{-1})$. The functors
of~$\k$-induction and~$\k$-restriction are exact functors
as~$J^1(\b,\La)$ is pro-$p$.

Now let~$[\Y,n_\Y,0,\b]$ be another self-dual semisimple stratum
with~$\MM_{\Y}=\MM_{\La}$ and~$\AA(\La)\subseteq\AA(\Y)$,
and let~$\th_\Y$ be the transfer
of~$\th$. Let~$\k$ be a~$\b$-extension and let~$\k_\Y$ be a
compatible~$\b$-extension of~$J(\b,\Y)$. 
Set~$P_{\La,\Y}^E=P(\La_E)/P^1(\Y_E)$, a parabolic subgroup
of~$M(\Y_E)$ with Levi factor~$M(\La_E)$; we
write~$i^{M(\Y_E)}_{P_{\La,\Y}^E}$ for the parabolic induction functor
and~$r^{M(\Y_E)}_{P_{\La,\Y}^E}$ for its adjoint. 
By transitivity of induction, an exercise shows that
we have isomorphisms of functors
\[
\I_{\k_\Y}\circ \,i^{M(\Y_E)}_{P_{\La,\Y}^E}\simeq \I_{\k}
\text{  and  } r^{M(\Y_E)}_{P_{\La,\Y}^E}\circ \R_{\k_\Y}\simeq \R_{\k},
\]
where the latter follows from the former by unicity of the adjoint.

We also have the special case of these functors when the stratum is
zero, which we can apply in~$\GE$. Thus, since~$\La_E$ is
an~$\o_E$-lattice chain, we have a level zero \emph{parahoric induction}
functor~$\I_{\La_E}:\RR_R(M(\La_E))\to \RR_R(\GE)$ attached
to~$[\La,n,0,\b]$ given by
\[
\I_{\La_E}(-)=\ind_{P(\La_E)}^{\GE}(\infl_{M(\La_E)}^{P(\La_E)}(-))
\]
where~$\infl_{M(\La_E)}^{P(\La_E)}:\RR_R(M(\La_E))\to\RR_R(P(\La_E))$
is the functor defined by trivial inflation to~$P^1(\La_E)$.  The
functor~$\I_{\La_E}$ possesses a right adjoint, which we call level
zero \emph{parahoric restriction},~$\R_{\La_E}:\RR_R(\GE)\to \RR_R(M(\La_E))$
given by the functor of~$P^1(\La_E)$-invariants 
\[
\R_{\La_E}(-)=(-)^{P^1(\La_E)},
\]
with the group~$P(\La_E)/P^1(\La_E)\simeq M(\La_E)$ acting
naturally. Level zero parahoric induction and restriction are exact
functors.

\section{Level zero interlude}\label{Levzerosection}

In this section we recall some results of Morris~\cite{Morris93} and Vign\'eras~\cite{vigneras} on level zero representations of~$G$ (cf. also \cite[\S 4]{MR1950479}). Later, we will apply them to~$\GE$, which will be a product of groups like~$G$ over extensions of~$F$.  The results of this section apply in the greater generality of~\cite{Morris93}, and we retain the notation of [ibid.] as it is much more convenient here, as such, the notation of this section is independent of that of the rest of the paper.  We recall this notation briefly below and explain how to translate to our notation in the rest of the paper.

Let~$\mathbb{G}$ be a connected reductive group over~$F$,~$\mathbb{T}$ be a maximal~$F$-split torus in~$\mathbb{G}$, and~$\mathbb{N}=N_{\mathbb{G}}(\mathbb{T})$.  We write~$G=\mathbb{G}(F)$,~$T=\mathbb{T}(F)$, and~$N=\mathbb{N}(F)$ for the respective groups of~$F$-points. Let~$B$ be an Iwahori subgroup of~$G$.  Following [ibid.],~$(G,B,N)$ is called a \emph{generalised affine~$BN$-pair}, and, associated to this data, we have a \emph{generalised affine Weyl group}~$W=N/B\cap N$.   According to [ibid.], we have a decomposition~$W=\Omega\ltimes W'$ with~$W'$ the affine Weyl group of some split affine root system.  Let~$S$ be a set of fundamental reflections in~$W'$.

If~$J\subset S$ is a proper subset of~$S$, we let~$W_J$ be the subgroup of~$W$ generated by the reflections in~$J$.  The standard parahoric subgroups of~$G$ correspond to proper subsets of~$S$, via~$J\subset S$ maps to~$P_J=B N_J B$ for~$N_J$ any set of representatives of~$W_J$ in~$G$.  Given a parahoric subgroup~$P_J$, we write~$U_J$ for its pro-$p$ unipotent radical and~$M_J=P_J/U_J$ the points of a connected reductive group over a finite field. We write~$U_B$ for the pro-$p$ unipotent radical of~$B=P_{\emptyset}$.  

Let~$J,K$ be proper subsets of~$S$.  A set of double coset representatives~$\overline{D}_{J,K}$ for~$W_J\backslash W/W_K$ is called \emph{distinguished} if each representative has minimal length in its double coset, (cf. [ibid., \S 3.10]).  A set of double coset representatives~$D_{J,K}$ for~$P_J\backslash G/P_K$ is called \emph{distinguished} if its projection to~$W$ is a set of distinguished double coset representatives for~$W_J\backslash W/W_K$.  Let~$D_{J,K}$ be a set of distinguished set of double coset representatives for~$P_J\backslash G/P_K$.  Let~$d\in D_{J,K}$ and~$w$ be its projection in~$W$. By~[ibid., Lemma 3.19, Corollary 3.20, Lemma 3.21], we have
\begin{enumerate}
\item\label{Part1levzero}~$P_{J\cap wK}=U_J(P_J\cap \presuper{n}P_K)$ with unipotent radical~$U_{J\cap wK}=U_J(P_J\cap \presuper{n}U_K)$.
\item~$P^J_{J\cap wK}=P_{J\cap wK}/U_J$ is a parabolic subgroup of~$M_J=P_J\cap U_J$.
\end{enumerate}
We can form the following lattice of groups:
\begin{center}
\begin{tikzpicture}
\matrix [matrix of math nodes, column sep=1cm, row sep=1cm,text height=2ex, text depth=0.25ex]
{
\node(top1) {1}; &\node(top2) {U_J};&\node(top3) {P_J}; &\node(top4) {M_J};&\node(top5) {1}; \\
\node(mid1) {1}; &\node(mid2) {U_J};&\node(mid3) {P_{J\cap wK}}; &\node(mid4) {P^J_{J\cap wK}};&\node(mid5) {1}; \\
\node(topp1) {1}; &\node(topp2) {U_J(P_J\cap \presuper{n}U_K)};&\node(topp3) {P_{J\cap wK}}; &\node(topp4) {M_{J\cap wK}};&\node(topp5) {1}; \\
};
\draw [->] (top1) -- (top2);
\draw [->] (top2) -- (top3);
\draw [->] (top3) -- (top4);
\draw [->] (top4) -- (top5);
\draw [->] (mid1) -- (mid2);
\draw [->] (mid2) -- (mid3);
\draw [->] (mid3) -- (mid4);
\draw [->] (mid4) -- (mid5);
\draw [->] (topp1) -- (topp2);
\draw [->] (topp2) -- (topp3);
\draw [->] (topp3) -- (topp4);
\draw [->] (topp4) -- (topp5);
\draw[double distance = 1.5pt] (top2) -- (mid2);
\draw[double distance = 1.5pt] (mid3) -- (topp3);
\draw[right hook->] (mid3) -- (top3);
\draw[right hook->] (mid4) -- (top4);
\draw[right hook->] (topp4) -- (mid4);
\draw[right hook->] (mid2) -- (topp2);
\end{tikzpicture}
\end{center}

Furthermore, as~$D_{J,K}^{-1}$ is a set of  distinguished double coset representatives for~$P_K\backslash G/P_J$, the group~$P^K_{w^{-1}J\cap K}$ is a parabolic subgroup of~$M_K$ and we can form an analogous diagram for~$P^K_{w^{-1}J\cap K}$.  Note also that~$M_{J\cap wK}= (M_{w^{-1}J\cap K})^{w}$.  

This section collects results based upon the following theorem of Vign\'eras.  Before we state it, we must recall the parahoric induction/restriction functors in this notation; let~$I_J:\mathfrak{R}_R(M_J)\rightarrow \mathfrak{R}_R(G)$ denote the parahoric induction functor
\[I_J(-)=\ind_{P_J}^G(\infl_{M_J}^{P_J}(-)),\]
and~$R_J:\mathfrak{R}_R(G)\rightarrow\mathfrak{R}_R(M_J)$ denote, its right adjoint, the parahoric restriction functor
\[R_J(-)=(-)^{U_J}.\]
The normaliser~$N_G(P_J)$ of~$P_J$ in~$G$ normalises~$U_J$, and~$M_J^+=N_G(P_J)/U_J$ contains~$M_J$ as a normal subgroup.  We write~$I_J^+:\mathfrak{R}_R(M_J^+)\rightarrow \mathfrak{R}_R(G)$ for the functor
\[I_J^+(-)=\ind_{N_G(P_J)}^G(-),\]
and~$R_J^+:\mathfrak{R}_R(G)\rightarrow\mathfrak{R}_R(M^+_J)$ its right adjoint, again given by~$U_J$-invariants.

\begin{thm}[{\cite[Basic decomposition~5.1]{vigneras}}]\label{vigneraslevzero}
We have an isomorphism of functors
\[
\R_{J}\circ\I_{K}\simeq \bigoplus_{w\in \overline{D}_{J,K}} i^{M_J}_{P^J_{J\cap wK}}\circ \( r^{M_K}_{P^K_{w^{-1}J\cap K}}\)^w.
\]
\end{thm}

\begin{corollary}\label{levelzerocorollaries}
Let~$\t$ be a cuspidal~$R$-representation of~$M_K$.
\begin{enumerate}
\item\label{part1levelzerocorollaries} The representation~$\R_{K}\circ\I_{K}(\t)$ is a direct sum of conjugates of~$\t$. 
\[\R_K\circ I_K(\t)\simeq \bigoplus_{\substack{w\in W_{K,K}\\ wK=K}}\t^w.\]
Moreover, if~$P_K$ is a maximal and~$\tau^+$ is an irreducible~$R$-representation of~$M_K^+$ with cuspidal restriction to~$M_K$, then~\[\R^+_{K}\circ\I^+_{K}(\t^+)=\t^+.\] 
\item\label{part2levelzerocorollaries} Suppose that~$P_K$ is maximal and~$P_J$ is not conjugate to~$P_K$ in~$G$. Then~\[\R_{J}\circ\I_{K}(\t)=0.\]
\end{enumerate}
\end{corollary}

\begin{proof}
All statements are straightforward applications of the theorem.  Part \ref{part1levelzerocorollaries} is \cite[Corollaries 5.2 \& 5.3]{vigneras}, and part \ref{part2levelzerocorollaries} follows as if~$P_J$ is not conjugate to~$P_K$, then~$P^K_{w^{-1}J\cap K}$ is a proper parabolic subgroup of~$M_K$, for any~$w\in W_{J,K}$, and hence~$r^{M_K}_{P^K_{w^{-1}J\cap K}}(\tau^w)=0$ by cuspidality.
\end{proof}

\begin{remark}
In case (i) of Lemma \ref{levelzerocorollaries}, the direct sum can be infinite.  Indeed this is the case when~$K$ is empty (and the building of~$G$ is not a point).
\end{remark}

Finally, we will need the following variant of~\cite[Proposition 4.13]{Morris93}, (cf.~\cite[Lemma 1.1]{St08}), which requires a different proof in our setting.
\begin{lemma}\label{Lemma11}
Let~$J,K$ be proper subsets of~$S$, and~$D$ be a set of distinguished double coset representatives for~$P_K\backslash G/P_J$.  Let~$\tau$ be a representation of~$M_J$ with cuspidal restriction to~$M_J^{\circ}$, and let~$n\in D$.  If~$n$ lies in the support of~$\mathcal{H}(G,\tau\mid_{U_B})$, i.e. 
\[\Hom_{U_B\cap U_B^n}(\tau,\tau^n)\neq 0,\]
then~$wK=J$, where~$w\in W$ is the projection of~$n$.
\end{lemma}

\begin{proof}
By~\cite[Lemma 3.21]{Morris93}, we have~$P_J\cap U_K^n\subseteq U_{J\cap wK}\subseteq U_B$.  Hence, as~$U_K^n\subseteq U_B^n$, we have
\[\Hom_{U_B\cap U_B^n}(\tau,\tau^n)\subseteq \Hom_{P_J\cap U_K^n}(\tau,\tau^n)=\Hom_{P_J\cap U_K^n}(\tau,\dim(\tau)1) .\]
But, by [ibid.],~$P_J\cap U_K^n$ is the unipotent radical of the parabolic subgroup~$P_{J\cap wK}/U_J$ of~$M_J$.  Hence by cuspidality of~$\tau$, we must have~$wK=J$. 
\end{proof}

\subsection{Level zero Hecke algebras}

Let~$P^{\circ}(\Y)$ be a parahoric subgroup of~$G$ associated to the~$\o_F$-lattice sequence~$\Y$ with pro-$p$ unipotent
radical~$P^1(\Y)$ and connected finite reductive
quotient~$M^{\circ}(\Y)$.  

\begin{remark}
By conjugating if necessary, we the parahoric subgroup $P^{\circ}(\Y)$ will be equal to a standard parahoric subgroup $P_J$ considered above, and we will interchange notations 
freely.
\end{remark}

Let~$Q^{\circ}(\La)$ be a parabolic
subgroup of~$M^{\circ}(\Y)$ with Levi
decomposition~$Q^{\circ}(\La)=M^{\circ}(\La)\ltimes U^{\circ}(\La)$,
and denote by~$P^{\circ}(\La)$ the parahoric subgroup which is the
preimage of~$Q^{\circ}(\La)$ under the projection
map~$P^{\circ}(\Y)\to M^{\circ}(\Y)$.  Thus the quotient
of~$P^{\circ}(\La)$ by its pro-$p$ unipotent radical~$P^1(\La)$
is~$M^{\circ}(\La)$.  Let~$\t$ be an irreducible cuspidal
representation of~$M^{\circ}(\La)$ and~$\widetilde{\t}$ denote both
its inflation to~$Q^{\circ}(\La)$ and to~$P^{\circ}(\La)$.  The
following Lemma follows immediately from the definitions.

\begin{lemma}
We have a support preserving isomorphism of Hecke
algebras~$\Hh(M^{\circ}(\Y),\widetilde{\t})\simeq \Hh(P^{\circ}(\La),
\widetilde{\t})$: if~$f\in\Hh(M^{\circ}(\Y),\widetilde{\t})$ is
supported on~$Q^{\circ}(\La)yQ^{\circ}(\La)$ for~$y\in M^{\circ}(\Y)$
then the corresponding element~$f'\in\Hh(P^{\circ}(\La),
\widetilde{\t})$ is supported on~$P^{\circ}(\La)yP^{\circ}(\La)$.
\end{lemma}

Let~$W(M^{\circ}(\La),\t)$ denote the \emph{inertia group} of~$\t$,
that is, the elements of the relative Weyl group of~$M^{\circ}(\La)$
in~$M^{\circ}(\Y)$ which normalize~$\t$
(see~\cite[Proposition~4.2.11]{MR2799052}).
We can give a presentation of the
algebra~$\Hh(M^{\circ}(\Y),\widetilde{\t})$ due to
Howlett--Lehrer~\cite{HL80} when~$R=\CC$ and to
Geck--Hiss--Malle~\cite{GHM96} in general.

\begin{thm}[{\cite[Theorem~4.2.12]{MR2799052}}]\label{finiteheckedescription}
There are a Coxeter system~$(W_1,S_1)$ and a finite group~$\Omega$
acting on~$(W_1,S_1)$ such that~$W(M^{\circ}(\La),\t)\simeq
\Omega\ltimes W_1$; furthermore~$\Hh(M^{\circ}(\Y),\widetilde{\t})$
has a basis~$\{T_w:w\in W(M^{\circ}(\La),\t)\}$ which gives a
presentation of the algebra with the following rules for
multiplication:
\begin{enumerate}
\item for all~$w\in W$ and~$w'\in\Omega$,
\[
T_w\star T_{w'}=\mu(w,w') T_{ww'}\quad\text{and}\quad 
T_{w'}\star T_{w}=\mu(w',w) T_{w'w},
\]
for some~$2$-cocycle~$\mu: W(M^{\circ}(\La),\t)\times
W(M^{\circ}(\La),\t)\to R^\times$;
\item for~$s\in S_1$, there are~$p_s\in R\backslash \{0,1\}$, such that, 
\[
T_s\star T_w=\begin{cases}T_{sw}&\text{if } l_1(sw)>l_1(w),\\
p_s T_{sw} +(p_s-1)T_w&\text{if } l_1(sw)<l_1(w),
\end{cases}
\]
for all~$s\in S_1$ and~$w\in W_1$, where~$l_1$ is the length function on~$W_1$.
\end{enumerate}
\end{thm}
\section{Reduction to level zero}\label{Redlevzerosect}

Let~$[\Y,n_{\Y},0,\b]$ and~$[\La,n_{\La},0,\b]$ be self-dual
semisimple strata in~$A$.  By conjugating by an element of~$\GE$, if
necessary, we assume that~$\Y_E$ and~$\La_E$ lie in the closure of a
common chamber in the building of~$\GE$, corresponding to an~$\mathfrak{o}_E$-lattice sequence~$\Gamma_E$ in~$V$.  As before,
let~$\th_{\Y}\in\Cc_{-}(\Y,0,\b)$
and~$\th_{\La}=\t_{\Y,\La,\b}(\th_{\Y})$.  Let~$\eta_{\Y}$ be the
unique Heisenberg representation containing~$\th_{\Y}$
and~$\eta_{\La}=\t_{\Y,\La,\b}(\eta_{\Y})$ the unique Heisenberg
representation containing~$\th_{\La}$.  Let~$\k_{\Y}$ be a
standard~$\b$-extension of~$\eta_{\Y}$ and~$\k_{\La}$ be a
standard~$\b$-extension of~$\eta_{\La}$.

We will abbreviate~$J_{\Y}=J(\b,\Y)$, and also~$P_{\Y}=P(\Y_E)$
and~$M_{\Y}=M(\Y_E)$, with analogous notation for~$\La$ and~$\Ga$. We
also write~$J^1_{\Ga,\Y}=P^1_\Ga J^1_\Y$, etc.

\begin{lemma}\label{intetakappa}
The intertwining of~$\eta_{\La}$ and~$\k_{\Y}$ in~$G$ is given by 
\[
\dim_R (\Hom_{J^1_{\La}\cap J_{\Y}^g}(\k_{\La},\k_{\Y}^g))=
\begin{cases} 1&\text{if }g\in J^1_{\Y}\GE J^1_{\La};\\ 
0&\text{otherwise.}\end{cases}
\]
\end{lemma}

\begin{proof}
We have~$J^1_{\La}\cap J^g_{\Y}\subseteq J^1_{\La}\cap K^g$ for some
Sylow~$p$-subgroup~$K$ of~$J_{\Y}$.  All Sylow~$p$-subgroups
of~$J_{\Y}$ are conjugate to~$J^1_{\Ga,\Y}$ so~$K=(J^1_{\Ga,\Y})^j$ for
some~$j\in J_{\Y}$. Thus~$\Res^{J_\Y}_{K}\k_{\Y}\simeq
\eta_{\Ga,\Y}^j$ and, as vector spaces, we have
\[
\Hom_{J^1_{\La}\cap J_{\Y}^g}(\k_{\La},\k_{\Y}^g)\simeq
\Hom_{J^1_{\La}\cap
  (J^1_{\Ga,\Y})^{jg}}(\eta_{\La},\eta_{\Ga,\Y}^{jg}).
\]
As~$\eta_{\Ga,\Y}$ extends~$\eta_{\Y}$, the result now follows by
applying Lemma~\ref{intetas2} and Theorem~\ref{intetas}.
\end{proof}
Let~$\t$ be a representation of~$M_{\Y}$ which we
identify with a representation of~$J_{\Y}$ trivial on~$J^1_{\Y}$ and
with a representation of~$P_\Y$ trivial on~$P^1_\Y$.  By Mackey's restriction-induction formula and exactness
of~$\k_{\La}$-restriction we have the following lemma.  

\begin{lemma}\label{lemmaMackey}
We have isomorphisms of representations of~$M_\La$
\begin{align*}
\R_{\k_{\La}}\circ\I_{\k_{\Y}}(\t)&\simeq 
\bigoplus_{J_{\Y}\backslash G/J_{\La}}\Hom_{J_{\La}^1}\(\k_{\La},\ind_{J_{\La}\cap J_{\Y}^g}^{J_{\La}}((\k_{\Y}\otimes\t)^g)\);\\
\R^E_{{\La}}\circ \I^E_{{\Y}}(\t)&\simeq  
\bigoplus_{P_{\Y}\backslash \GE/P_{\La}}\Hom_{P_\La^1}\(1,\ind_{P_\La\cap P_\Y^g}^{P_\La}(\t^g)\).
\end{align*}
\end{lemma}

\begin{lemma}\label{zeroif}
Let~$g\in G$.  If~$\Hom_{J_{\La}^1}\(\k_{\La},\ind_{J_{\La}\cap
  J_{\Y}^g}^{J_{\La}}(\k_{\Y}\otimes\t)^g\)\neq0$ then~$g\in J_{\Y}^1\GE J_{\La}^1$.
\end{lemma}

\begin{proof}
Consider~$\Hom_{J_{\La}^1}\(\k_{\La},\ind_{J_{\La}\cap J_{\Y}^g}^{J_{\La}}(\k_{\Y}\otimes\t)^g\)$ as an abstract~$R$-vector space.  We have
\[
\Hom_{J_{\La}^1}\(\Res^{J_{\La}}_{J_{\La}^1}\k_{\La},\Res^{J_{\La}}_{J_{\La}^1}\circ\ind_{J_{\La}\cap J_{\Y}^g}^{J_{\La}}(\k_{\Y}\otimes\t)^g\)\simeq 
\bigoplus_{h\in(J_{\La}\cap J_{\Y}^g)\backslash J_{\La}/J_{\La}^1}\Hom_{J_{\La}^1\cap J_{\Y}^{gh}}\(\eta_\La,(\k_{\Y}\otimes\t)^{gh}\)
\]
by Mackey's restriction-induction formula and Frobenius reciprocity.  We have an injection of vector spaces
\[
\Hom_{J_{\La}^1\cap J_{\Y}^{gh}}\(\eta_{\La},(\k_{\Y}\otimes\t)^{gh}\)
\hookrightarrow 
\Hom_{J_{\La}^1\cap (J_{\Y}^1)^{gh}}\(\eta_{\La},(\k_{\Y}\otimes\t)^{gh}\)
\]
and on~$(J_{\Y}^1)^{gh}$ we have~$\k_{\Y}^{gh}=\eta^{gh}_{\Y}$
and~$\t^{gh}$ is a multiple of the trivial representation. Thus~$gh\in\I_G(\eta_{\La},\eta_{\Y})=J^1_{\Y}\GE J^1_{\La}$,
by Theorem~\ref{intetas}, and we deduce 
that~$g\in J^1_{\Y}\GE J_{\La}= J^1_{\Y}\GE J^1_{\La}$.
\end{proof}

\begin{lemma}\label{lemmaifzerothen}
\begin{enumerate}
\item Let~$g\in \GE$. If~$\Hom_{P^1_{\Lambda}}(1,\ind^{P_{\Lambda}}_{P_{\Lambda}\cap P_{\Upsilon}^g}(\tau^g))=0$, then~$\Hom_{J_{\Lambda}^1}(\kappa_{\Lambda},\ind^{J_{\Lambda}}_{J_{\Lambda}\cap J_{\Upsilon}^g}(\kappa_{\Upsilon}\otimes\tau)^g)=0$.
\item  As representations of~$M(\Y_E)$, we have isomorphisms
\[\Hom_{J^1_{\Y}}(\kappa_{\Y},\kappa_{\Upsilon}\otimes\tau)\simeq \Hom_{P^1_{\Y}}(1,\t)\simeq \t.\]
\end{enumerate}
\end{lemma}

\begin{proof}
As an abstract vector space, by Mackey theory, we have
\begin{align*}
\Hom_{J_{\Lambda}^1}(\kappa_{\Lambda},\ind^{J_{\Lambda}}_{J_{\Lambda}\cap J_{\Upsilon}^g}(\kappa_{\Upsilon}\otimes\tau)^g)&\simeq \bigoplus_{h\in (J_{\La}\cap J_{\Y}^g)\backslash J_{\La}/J_{\La}^1}\Hom_{J^1_{\La}\cap J_{\Y}^{gh}}(\eta_{\La},(\kappa_{\Y}\otimes\t)^{gh}).
\end{align*}
By Lemma \ref{intetakappa}~$gh$ intertwines~$\eta_{\La}$ with~$\kappa_\Y$ for every~$h\in J_{\La}$.  Hence by Lemma \ref{BK532} (applied with~$X_1=X_1^1=J_{\La}^1$,~$X_2=J_{\Y}^{gh}$,~$X_2^1=(J_\Y^1)^{gh}$,~$\mu_1=\eta_{\La}$,~$\mu_2=\kappa^{gh}$,~$\zeta_1=1$, and~$\zeta_2=\tau^{gh}$)  for each summand, we have an isomorphism of vector spaces  
\[\Hom_{J^1_{\La}\cap J_{\Y}^{gh}}\left(\eta_{\La},(\kappa_{\Y}\otimes\t)^{gh}\right)\simeq \Hom_{J_{\La}^1\cap J_{\Y}^{gh}}(1, \t^{gh}).\]
Moreover, as~$J_{\La}^1\cap J_{\Y}^{gh}$ contains~$P^1_{\La}\cap P_{\Y}^{gh}$, we have
\begin{align*}
\Hom_{J_{\La}^1\cap J_{\Y}^{gh}}(1, \t^{gh})\subseteq \Hom_{P^1_{\La}\cap P_{\Y}^{gh}}(1,\t^{gh})
\end{align*}
But, the right hand side is isomorphic as a vector space to a direct summand of the representation
\[\Hom_{P_{\Lambda}^1}(1,\ind^{P_{\Lambda}}_{P_{\Lambda}\cap P_{\Upsilon}^g}(\tau)^g)\simeq \bigoplus_{h'\in(P_{\La}\cap P_{\Y}^g)\backslash P_{\La}/P_{\La}^1}\Hom_{P^1_{\La}\cap P_{\Y}^{gh'}}(1,\t^{gh'}),\]
where the above decomposition is again an isomorphism of abstract vector spaces obtained by Mackey theory.  However, by our hypotheses~$\Hom_{P^1_{\Lambda}}(1,\ind^{P_{\Lambda}}_{P_{\Lambda}\cap P_{\Upsilon}^g}(\tau^g))$ is trivial, whence all the summands~$\Hom_{J^1_{\La}\cap J_{\Y}^{gh}}(\eta_{\La},(\kappa_{\Y}\otimes\t)^{gh})$ are trivial and, thus, so is~$\Hom_{J_{\Lambda}^1}(\kappa_{\Lambda},\ind^{J_{\Lambda}}_{J_{\Lambda}\cap J_{\Upsilon}^g}(\kappa_{\Upsilon}\otimes\tau)^g)$ and we have shown case (i).

For the second part, we can take~$S\in\Hom_{J_{\Y}}(\k_{\Y},\k_{\Y})$ to be the identity element.  By Lemma~\ref{BK532} (applied with~$X_1=X_1^1=X_2=X_2^1=J_{\Y}^1$,~$\mu_1=\mu_2=\eta_{\Y}$,~$\zeta_1=1$, and~$\zeta_2=\tau\mid_{J^1}=\dim(\tau) 1$) we have an isomorphism of vector spaces~$\Hom_{J_{\Y}^1}(1,\tau)\to \Hom_{J_{\Y}^1}(\k_{\Y},\k_{\Y}\otimes \tau)$ given by~$T\mapsto S\otimes T$.  The action of~$M(\Y_E)$ on~$\Hom_{J_{\Y}^1}(\k_{\Y},\k_{\Y}\otimes \s)$ induced from the action of~$M(\Y_E)$ on~$\R_{\k_{\Y}}\circ \I_{\k_{\Y}}(\tau)$ is given by~$m\cdot \Phi=\k_\Y\otimes\t(j)\circ \Phi\circ\k_\Y(j^{-1})$, for~$m\in M(\Y_E)$,~$\Phi\in \Hom_{J_{\Y}^1}(\k_{\Y},\k_{\Y}\otimes \tau)$ and~$j$ any representative of~$m$ in~$J$.  Thus, we have
\begin{align*}
m\cdot S\otimes T&=\k\otimes\tau(j)\circ(S\otimes T)\circ \k(j^{-1})\\
&=\k(j)\circ S\circ\k(j^{-1})\otimes \tau(j)\circ T.
\end{align*}
However, as~$S\in\Hom_{J}(\k_{\Y},\k_{\Y}^g)$, whence~$\k_{\Y}(j)\circ S\circ \k_{\Y}(j^{-1})= S$.  Therefore, we have~$m\cdot S\otimes T=S\otimes m\cdot T$, the isomorphism of vector spaces is an isomorphism of representations of~$M(\Y_E)$.  Moreover,
\[\Hom_{J^1_{\Y}}(1,\t)\simeq \Hom_{P^1_{\Y}}(1,\t)\simeq \t.\]
%
\end{proof}
%

\begin{corollary}\label{maincorollary}
Let~$\t$ be a representation of~$M_{\Y}$.
\begin{enumerate}
\item\label{maincorollarypart1} If~$\R_{{\La}}^E\circ \I_{{\Y}}^E(\t)$ is trivial then so is~$\R_{\k_{\La}}\circ\I_{\k_{\Y}}(\t)$.
\item\label{maincorollarypart2} Suppose~$\t$ is irreducible with cuspidal restriction to~$M^{\circ}_{\Y}$.  If~$\GE$ has compact centre and~$P^{\circ}(\Y_E)$ is a maximal parahoric subgroup of~$\GE$ then
\[\R_{\k_{\Y}}\circ\I_{\k_{\Y}}(\t)\simeq  \t.\]
\end{enumerate}
\end{corollary}

\begin{proof}
By Lemmas \ref{lemmaMackey} and \ref{zeroif}, we have isomorphisms of representations of~$M_\La$
\begin{align*}
\R_{\k_{\La}}\circ\I_{\k_{\Y}}(\t)&\simeq 
\bigoplus_{J_{\Y}\backslash J_{\Y}^1\GE J_{\La}^1/J_{\La}}\Hom_{J_{\La}^1}\(\k_{\La},\ind_{J_{\La}\cap J_{\Y}^g}^{J_{\La}}((\k_{\Y}\otimes\t)^g)\);\\
\R^E_{{\La}}\circ \I^E_{{\Y}}(\t)&\simeq  
\bigoplus_{P_{\Y}\backslash \GE/P_{\La}}\Hom_{P_\La^1}\(1,\ind_{P_\La\cap P_\Y^g}^{P_\La}(\t^g)\).
\end{align*}
We choose a set of distinguished double coset representatives for~$P_{\Y}\backslash \GE/P_{\La}$, which by with the bijection of Lemma \ref{doublecosets}, fixes a set of double coset representatives of~$J_{\Y}\backslash J_{\Y}^1\GE J_{\La}^1/J_{\La}$ in~$\GE$.  We can now compare the summands of both isomorphisms on the right.  Part (i) follows from Lemma~\ref{lemmaifzerothen} part (i), and Lemma \ref{levelzerocorollaries} Part (ii). 
 For Part (ii) notice by Lemma~\ref{lemmaifzerothen} parts (i) and (ii), and Lemma \ref{levelzerocorollaries} Part (i), that the only summands which contribute correspond to distinguished double cosets~$P_{\Y}n P_{\Y}$ where~$n$ has projection~$w$ in the extended affine Weyl group satisfying~$wK=K$ for~$K$ the proper subset of fundamental reflections of the affine Weyl group corresponding to~$P^{\circ}_{\Y}$.  However, as~$P^{\circ}_{\Y}$ is maximal~$wK=K$ implies that~$n\in N_{G_E}(P^{\circ}_{\Y})=P_{\Y}$ by \cite[Appendix]{Morris93}.  Thus, Part (ii) follows from Lemma~\ref{lemmaifzerothen} Part (ii). 
\end{proof}

\section{Skew covers}\label{skewcoverssect}

This section is concerned with revisiting and making the necessary changes to the second authors construction of covers in~\cite{St08} so that the \emph{same} construction works in positive characteristic~$\ell$.  The construction follows \emph{mutatis mutandis} the constructions of the second author for complex representations and rather than go through all the proofs, which are lengthy, we introduce all the the notation of \emph{op. cit.} and indicate where changes need to be made to the proofs.

\subsection{Iwahori decompositions}

Let~$[\La,n,0,\b]$ be a semisimple stratum with associated splitting~$V=\bigoplus_{i\in I} V^i$.  A decomposition~$V=\bigoplus_{j=1}^m W^{(j)}$ of~$V$ is called \emph{subordinate} to~$[\La,n,0,\b]$ if \begin{enumerate}
\item each~$W^{(j)}\cap V^i$ is an~$E_i$-subspace of~$V^i$;
\item~$W^{(j)}=\bigoplus_{i\in I}(W^{(j)}\cap V^i)$;
\item~$\La(r)=\bigoplus_{j=1}^m(\La(r)\cap W^{(j)})$, for all~$r\in \ZZ$;
\end{enumerate}
It is called \emph{properly subordinate} to~$[\La,n,0,\b]$ if it is subordinate and, also,
\begin{enumerate}\setcounter{enumi}{3}
\item for each~$r\in \ZZ$ and~$i\in I$, there is at most one~$j$ such that
\[
(\La(r)\cap W^{(j)}\cap V^i))\supsetneq (\La(r+1)\cap W^{(j)}j\cap V^i).
\]
\end{enumerate}

If~$[\La,n,0,\b]$ is a semisimple stratum and~$V=\bigoplus_{j=1}^m W^{(j)}$ is a decomposition which is subordinate to~$[\La,n,0,\b]$ then we put~$\La^{(j)}$ to be the~$\o_{E}$-lattice sequence in~$W^{(j)}$ given by~$\La^{(j)}(r)=\La(r)\cap W_j$ and put~$\b^{(j)}=e^{(j)}\b e^{(j)}$ where~$e^{(j)}$ is the orthogonal projection~$V\to W_j$.  Then there is an integer~$n^{(j)}$ such that~$[\La^{(j)},n^{(j)},0,\b^{(j)}]$ is a semisimple stratum in~$A^{(j)}=\End_F(W^{(j)})$ with splitting~$W^{(j)}=\bigoplus_{i\in I}(W^{(j)}\cap V^i)$.  We put~$B^{(j)}=C_{A^{(j)}}(\b^{(j)})$.  

Let~$\tM$ denote the Levi subgroup of~$\tG$ equal to the stabiliser of the decomposition~$V=\bigoplus_{j=1}^m W^{(j)}$ and let~$\tP$ be any parabolic subgroup of~$\tG$ with Levi factor~$\tM$ and Levi decomposition~$\tP=\tM\ltimes \tU$.

\begin{lemma}[{\cite[Propositions~$5.2$ and~$5.4$]{St08}}]
If~$V=\bigoplus_{j=1}^m W^{(j)}$ is subordinate to~$[\La,n,0,\b]$ then~$\tJ^1(\b,\La)$ and~$\tH^1(\b,\La)$ have Iwahori decompositions with respect to~$(\tM,\tP)$.  Moreover
\[
\tH^1(\b,\La)\cap \tM=\prod_{j=1}^m\tH^1(\b^{(j)},\La^{(j)}),
\]
there is a similar decomposition for~$\tJ^1(\b,\La)\cap\tM$, and we can form the groups 
\[
\tH^1_{\tP}=\tH^1(\b,\La)(\tJ^1(\b,\La)\cap\tU),\quad \tJ^1_{\tP}=\tH^1(\b,\La)(\tJ^1(\b,\La)\cap\tP)
\]
which have Iwahori decompositions with respect to any parabolic subgroup with Levi factor~$\tM$. If the decomposition~$V=\bigoplus_{j=1}^m W^{(j)}$ is properly subordinate to~$[\La,n,0,\b]$ then~$\tJ(\b,\La)$ also has an Iwahori decomposition with respect to~$(\tM,\tP)$, we also have
\[
\tJ(\b,\La)\cap \tM=\prod_{j=1}^m\tJ(\b^{(j)},\La^{(j)}),
\]
and we can form the group~$\tJ_{\tP}=\tH^1(\b,\La)(\tJ(\b,\La)\cap\tP)$ which has an Iwahori decomposition with respect to any parabolic subgroup with Levi factor~$\tM$.
\end{lemma}

Let~$[\La,n,0,\b]$ be a self-dual semisimple stratum.  A decomposition~$V=\bigoplus_{j=-m}^m W^{(j)}$ is called self-dual if, for~$-m\< j\< m$, the orthogonal complement of~$W^{(j)}$ is~$\bigoplus_{k\neq \pm j} W^{(k)}$.  Put~$M=\tM\cap G$ a Levi subgroup of~$G$ and~$M^+=\tM\cap G^+$ a Levi subgroup of~$G^+$. Choosing a~$\s$-stable parabolic subgroup~$\tP$ of~$G$ with Levi factor~$\tM$, we have~$P=\tP\cap G$ a parabolic subgroup of~$G$ with Levi factor~$M$ and~$P^+=\tP\cap G^+$ a parabolic subgroup of~$G^+$ with Levi factor~$M^+$.

\begin{lemma}[{\cite[Corollaries~5.10 and~5.11]{St08} (\cf~\cite[Fait~8.10]{Finitude})}]\label{lemma72}
If~$V=\bigoplus_{j=-m}^m W^{(j)}$ is a self-dual subordinate decomposition to~$[\La,n,0,\b]$, then the groups~$H^1(\b,\La)$ and~$J^1(\b,\La)$ have Iwahori decompositions with respect to~$(M,P)$,
\[
H^1(\b,\La)\cap M\simeq H^1(\b^{(0)},\La^{(0)})\times\prod_{j=1}^m\tH^1(\b^{(j)},\La^{(j)}),
\]
there is a similar decomposition for~$J^1(\b,\La)$, and we can form the groups 
\[
H^1_{P}=H^1(\b,\La)(J^1(\b,\La)\cap U),\quad J^1_{P}=H^1(\b,\La)(J^1(\b,\La)\cap P).
\]
Moreover, if the decomposition is properly subordinate to~$[\La,n,0,\b]$ then~$J^+(\b,\La)$ has an Iwahori decomposition with respect to~$(M^+,P^+)$,~$J(\b,\La)$ and~$J^{\circ}(\b,\La)$ have Iwahori decompositions with respect to~$(M,P)$,
\[
J(\b,\La)\cap M\simeq J(\b^{(0)},\La^{(0)})\times\prod_{j=1}^m\tJ(\b^{(j)},\La^{(j)}),
\]
there are similar decompositions for~$J^+(\b,\La)\cap M^+$ and~$J^{\circ}(\b,\La)\cap M$, and we can form the groups
\[
J^+_{P}=H^1(\b,\La)(J^+(\b,\La)\cap  P),\quad J_{P}=H^1(\b,\La)(J(\b,\La)\cap  P),\quad J^{\circ}_{P}=H^1(\b,\La)(J^{\circ}(\b,\La)\cap  P).
\] 
\end{lemma}

Let~$\tth\in\Cc(\La,n,0,\b)$ and~$\teta$ be the unique Heisenberg representation of~$\tJ^1(\b,\La)$ containing~$\tth$.  By Lemma~\cite[Lemma~5.6]{St08}, we can define a character of~$\tH^1_{\tP}$ by
\[
\tth_{\tP}(hj)=\tth(h),
\]
for~$h\in\tH^1(\b,\La)$ and~$j\in\tJ^1(\b,\La)\cap \tU$.  

\begin{lemma}[{\cite[Corollary~5.7 and Lemma~5.8]{St08}}]
There exists a unique irreducible representation of~$\tJ^1_{\tP}$ containing~$\tth_{\tP}$.  Moreover~$\teta=\ind_{\tJ^1_{\tP}}^{\tJ^1(\b,\La)}(\teta_{\tP})$ and for each~$y\in\tGE$, there is a unique~$(\tJ^1_{\tP},\tJ^1_{\tP})$-double coset in~$\tJ^1(\b,\La)y\tJ^1(\b,\La)$ which intertwines~$\teta_{\tP}$ and~$I_{\tG}(\tth_{\tP})=I_{\tG}(\teta_{\tP})=\tJ^1_{\tP}\tGE\tJ^1_{\tP}$.
\end{lemma}

Let~$\th\in\Cc_{-}(\La,0,\b)$ and~$\eta$ be the unique Heisenberg representation of~$J^1(\b,\La)$ containing~$\th$.  We can define a character~$\th_P$ of~$H^1_P$ by 
\[
\th_P(hj)=\th(h),
\]
for~$h\in H^1(\b,\La)$ and~$j\in J^1(\b,\La)\cap U$.  Then~$\th_P=\mathbf{g}(\tth_{\tP})$ is the Glauberman transfer of~$\tth_{\tP}$ (as~$\tth_{\tP}$ is a character the Glauberman transfer here is just restriction to~$H^1_P$). 

We let~$\eta_P=\mathbf{g}(\teta_{\tP})$. Using properties of the Glauberman correspondence the following Lemma is proved in~\cite{St08}.

\begin{lemma}
The representation~$\eta_P$ is the unique irreducible representation of~$J^1_P$ which contains~$\th_P$,~$\eta=\ind_{J^1_P}^{J^1(\b,\La)}(\eta_P)$.  Moreover for each~$y\in\GE$, there is a unique~$(J^1_P,J^1_P)$-double coset in~$J^1(\b,\La)yJ^1(\b,\La)$ which intertwines~$\eta_P$ and~$\dim_R(I_g(\eta_P))$ is~$1$ if~$g\in J^1_P \GE^+ J^1_P$ and~$0$ otherwise.
\end{lemma}

Let~$\k^+$ be a standard~$\b$-extension of~$\eta$.  We can form the natural representation~$\k_P^+$ of~$J^+_P$ on the space of~$(J^1\cap U)$-fixed vectors in~$\k^+$ by normality.  Then~$\Res^{J^+_P}_{J^1_P}(\k^+_P)=\eta_P$, hence~$\k_P$ is irreducible.  The Mackey restriction formula as in~\cite[Proposition 5.13]{St08} shows that~$\ind_{J^+_P}^{J^+}(\k_P^+)\simeq \k^{+}$.     We can also define representations of~$\k_P$ of~$J_P$ and~$\k^{\circ}_P$ of~$J^{\circ}_P$, for which analogous statements hold and~$\Res^{J^+_P}_{J_P}(\k_P^+)=\k_P$,~$\Res^{J^+_P}_{J^{\circ}_P}(\k_P^+)=\k_P^{\circ}$. 

In the next Lemma we identify~$H^1(\b,\La)\cap M$ with~$H^1(\b^{(0)},\La^{(0)})\times\prod_{j=1}^m\tH^1(\b^{(j)},\La^{(j)})$ using Lemma~\ref{lemma72}, and use the similar identifications for~$J^1(\b,\La)\cap M$ and~$J(\b,\La)\cap M$.

\begin{lemma}[{\cite[Section~5]{St08}}]\label{lemma75}
If~$V=\bigoplus_{j=-m}^m W^{(j)}$ is a self-dual subordinate decomposition, then
\[
\th_P\mid_{H^1(\b,\La)\cap M}=\th^{(0)}\otimes\bigotimes_{j=1}^m\(\tth^{(j)}\)^2,
\]
with~$\th^{(0)}\in\Cc_{-}(\La^{(0)},0,\b^{(0)})$ and~$\tth^{(j)}\in\Cc(\La^{(j)},0,\b^{(j)})$.  Similarly we have 
\[
\eta_P\mid_{J^1(\b,\La)\cap M}=\eta^{(0)}\otimes\bigotimes_{j=1}^m\teta^{(j)},
\]
where~$\eta^{(0)}$ is the unique irreducible representation of~$J^1(\b^{(0)},\La^{(0)})$ containing~$\th^{(0)}$ and~$\teta^{(j)}$ is the unique irreducible representation of~$\tJ^1(\b^{(j)},\La^{(j)})$ containing~$\(\tth^{(j)}\)^2$.  Moreover, if~$V=\bigoplus_{j=-m}^m W^{(j)}$ is a self-dual properly subordinate decomposition,
\[
\k_P\mid_{J(\b,\La)\cap M}=\k^+_{(0)}\otimes\bigotimes_{j=1}^m\tk_{(j)},
\]
with~$\k^+_{(0)}$ an extension of~$\eta^{(0)}$ to~$J^+(\b^{(0)},\La^{(0)})$ and~$\tk_{(j)}$ an extension of~$\teta_{(j)}$ to~$\tJ(\b^{(j)},\La^{(j)})$.
\end{lemma}

\begin{lemma}[{\cite[Lemma~6.1]{St08}}]\label{Lemma61St}
Let~$K$ be a compact open subgroup of~$J^+(\b,\La)$ containing~$J^1(\b,\La)$ which has an Iwahori decomposition with respect to~$(M^+,P^+)$ with~$K\cap M^+=K^{(0)}\times\prod\tK^{(j)}$.  Let~$\rho$ be the inflation to~$K$ of an irreducible representation of~$K/J^1(\b,\La)$,~$\l=\Res^{J^+(\b,\La)}_K(\k^+)\otimes \rho$ and~$\l_P$ the representation of~$K_P=H^1(\b,\La)(K\cap P)$ on the space of~$J^1(\b,\La)\cap U$-fixed vectors in~$\l$.  Then
\begin{enumerate}
\item~$\l_P$ is irreducible and~$\l=\ind_{K_P}^K\l_P$.\label{Lem6.1part1}
\item~$\l_P\simeq \k_P\otimes \rho$ considering~$\rho$ as a representation of~$K_P/J^1_P\simeq K/J^1(\b,\La)$.\label{Lem6.1part2}
\item~$\l_P\mid_{K\cap M}=\l^{(0)}\otimes \bigotimes_{j=1}^m\tl^{(j)}$, where~$\tl^{(0)}=\k_{(0)}\mid_{K^{(0)}}\otimes\rho^{(0)}$ is a representation of~$K^{(0)}$ and~$\tl^{(j)}=\tk^+_{(j)}\mid_{\tK^{(j)}}\otimes \trho^{(j)}$ is a representation of~$\tK^{(j)}$, for~$1\< j\< m$.\label{Lem6.1part3}
\item There is a support preserving algebra homomorphism~$\Hh(G^+,\l_P)\simeq \Hh(G^+,\l)$; if~$\phi\in\Hh(G^+,\l)$ has support~$KyK$ for some~$y\in \GE^+$ then the corresponding~$\phi_P\in\Hh(G^+,\l_P)$ has support~$K_PyK_P$.
\end{enumerate}
\end{lemma}

\begin{proof}
The proof follows \emph{mutatis mutandis} the proof of~\cite[Lemma 6.1]{St08}, making use of the results quoted in this section and Lemma~\ref{simpcrit3} for parts~\ref{Lem6.1part1}, \ref{Lem6.1part2}, and \ref{Lem6.1part3}.
\end{proof}
 
The self-dual decomposition~$V=\bigoplus_{j=-m}^m W^{(j)}$ is \emph{exactly subordinate} to~$[\La,n,0,\b]$, in the sense of~\cite[Definition 6.5]{St08}, if~$P^{\circ}(\La_E)\cap M$ is a maximal parahoric subgroup of~$\GE\cap M$ and, for each~$j\neq 0$, there is an~$i$ such that~$W^{(j)}$ is contained in~$V^i$ and~$\AA(\La^{(j)})\cap B^{(j)}$ is a maximal~$\o_{E^i}$-order in~$B^{(j)}$, or equivalently, if it is minimal amongst all self-dual decompositions which are properly subordinate to~$[\La,n,0,\b]$.
 
For the rest of this section, we suppose that the self-dual decomposition~$V=\bigoplus_{j=-m}^m W^{(j)}$ is \emph{exactly subordinate} to~$[\La,n,0,\b]$. For~$j,k>0$, in~\cite[Section 6.2]{St08} a collection of Weyl group element~$s_{j,k} ,s_j,$ and~$s_j^{\varpi}$, all of which lie in~$\GE^+$, of~$G$ is defined.  The element~$s_{j,k}$ exchanges the blocks~$e^{(j)}Ae^{(j)}$ and~$e^{(k)}Ae^{(k)}$, and the blocks~$e^{(-j)}Ae^{(-j)}$ and~$e^{(-k)}Ae^{(-k)}$.  The elements~$s_j$ and~$s_j^{\varpi}$ exchange the blocks~$e^{(j)}Ae^{(j)}$ and~$e^{(-j)}Ae^{(-j)}$.  Let~$\La^M$ be a~$\o_E$-lattice sequence in~$V$ such that~$\AA(\La^M_E)$ is a maximal~$\o_E$-order containing~$\AA(\La_E)$.  For~$j,k>0$,~$W^{(j)}$ and~$W^{(k)}$ are called \emph{companion with respect to~$\La^M$} if~$s_{j,k}\in P^+(\La^M_E)$, while~$W^{(j)}$ and~$W^{(-j)}$ are called \emph{companion with respect to~$\La^M$} if~$s_{j}$ or~$s_j^{\varpi}$ lies in~$P^+(\La^M_E)$.  Following these definitions in \emph{op. cit.} an involution~$\s_j$ is defined on~$\tG_j=\{(\ov{g}^{-1},g)\in \tG^{(-j)}\times \tG^{(j)}\}$ by conjugation by~$s_j$.  Furthermore, by~\cite[Lemma 6.9]{St08}, the group~$\tJ(\b^{(j)},\La^{(j)})$ is stable under~$\s_j$, and , if~$1\< j<k\< m$ and~$W^{(j)}\simeq W^{(k)}$ as~$E_i$-spaces for some~$i$, then conjugation by~$s_{j,k}$ induces an isomorphism~$\tJ(\b^{(j)},\La^{(j)})\simeq \tJ(\b^{(k)},\La^{(k)})$.

\begin{lemma}[{\cite[Proposition 6.3, Corollary 6.10]{St08}}]
Suppose the self-dual decomposition~$V=\bigoplus_{j=-m}^m W^{(j)}$ is \emph{exactly subordinate} to~$[\La,n,0,\b]$.  Then~$\k^+_{(0)}$ is a standard~$\b^{(0)}$-extension of~$\eta^{(0)}$ to~$J^+(\b^{(0)},\La^{(0)})$ and~$\tk_{(j)}$ is a standard~$2\b^{(j)}$-extension of~$\teta^{(j)}$ to~$\tJ(\b^{(j)},\La^{(j)})$.  Furthermore, for~$1\< j\< m$, conjugation by~$s_j$ induces and equivalence~$\tk_{(j)}\circ \s_j\simeq \tk_{(j)}$, and, if~$1\< j<k\< m$ and~$W^{(j)}\simeq W^{(k)}$ as~$E_i$-spaces for some~$i$, then conjugation by~$s_{j,k}$ induces an equivalence~$\tk_{(j)}\simeq\tk_{(k)}$.
\end{lemma}

This lemma together with the comparison of~$\b$-extensions leads to the following observation, as in \emph{op. cit.}  Let~$\La^M,\La^{M'}$ be self-dual~$\o_E$-lattice sequences such that the associated~$\o_E$-orders are maximal and contain~$\AA(\La_E)$.  Let~$\k$ be a~$\b$-extension of~$\eta$ relative to~$\La^M$ and~$\k'$ be a~$\b$-extension of~$\eta$ relative to~$\La^{M'}$.  There are~$\s_i$-invariant characters~$\tchi^{(j)}$ of~$k_{E_i}^\times$ and a character~$\chi^{(0)}$ of~$M^\circ(\La_E)$ such that, setting~$\chi=\chi^{(0)}\otimes\bigotimes_{j=1}^m\tchi^{(j)}\circ\det^{(j)}$, we have
\[\k'=\ind_{J_P^\circ}^{J^{\circ}(\b,\La)}(\k_P\otimes \chi).\]

\subsection{$\k_P$-induction and restriction}

We have functors~$\I_{\k_P}:\RR_R(M^{\circ}(\La_E))\to \RR_R(G)$ and~$\I_{\k_L}:\RR_R(M^{\circ}(\La_E))\to \RR_R(L)$ with right adjoint functors~$\R_{\k_P}:\RR_R(G)\to \RR_R(M^{\circ}(\La_E))$ and~$\R_{\k_L}:\RR_R(L)\to \RR_R(M^{\circ}(\La_E))$; defined analogously to~$\I_{\k}$ and~$\R_{\k}$ in Section~\ref{Functors}.  In fact, as~$\ind_{J_P}^J(\k_P)\simeq \k$, we have natural isomorphisms of functors~$\I_{\k}\simeq \I_{\k_P}$ and~$\R_{\k}\simeq \R_{\k_P}$.

\subsection{Bounding~$I_G(\k_P)$}

Suppose~$P^{\circ}(\La_E)$ is not maximal.  Let~$N_E$ denote the normaliser in~$\GE$ of the product of maximal~$E_i$-split tori~$T_{E_i}$ in~${\GE}_i$, chosen relative to a certain~$E_i$-basis of~$V^i$ as in~\cite[Section 6]{St08}.  Let~$N_{\La}=\{w\in N_E:w\text{ normalises } P^{\circ}(\La_E)\cap M\}$ and~$N_{\La}(\rho)=\{n\in N_{\La}: \rho^n\simeq \rho\}$.

\begin{lemma}[{\cite[Corollary~6.16]{St08}}]\label{Cor616St}
The intertwining of~$\k_P^\circ$ is given by 
\[
I_G(\k_P^\circ)\supseteq J_P^\circ N_{\La}(\rho) J_P^{\circ},
\]
and the intertwining of~$\l_P^{\circ}=\l_P\mid_{J^{\circ}_P}$ is given by 
\[
I_G(\l_P^{\circ})= J^{\circ}_P N_{\La}(\rho) J^{\circ}_P.
\]
\end{lemma}

The proof follows exactly as in \emph{op. cit.} with one caveat: we
replace the use of~\cite[Proposition~1.1]{St08} with 
Lemma \ref{Lemma11}. 

\subsubsection{A Hecke algebra injection}

Let~$[\La,n,0,\b]$ and~$[\La',n',0,\b]$ be skew semisimple strata with~$\AA(\La_E)\subseteq \AA(\La'_E)$.  Let~$\th\in \Cc_{-}(\La,0,\b)$ and~$\th'=\t_{\La,\La',\b}(\th)$ be semisimple characters,~$\k$ and~$\k'$ compatible~$\b$-extensions of~$\th$ and~$\th'$, and~$\rho$ denote the inflation of an irreducible cuspidal representation~$\ov{\rho}$ of~$M^{\circ}(\b,\La)$ to the groups~$J^\circ(\b,\La)$,~$J^{\circ}_{\La,\La'}$ and~$P^{\circ}(\La_E)$.  We put~$\l=\k\otimes \rho$ and~$\l'=\k'\mid_{J^{\circ}_{\La,\La'}}\otimes\rho$.  We have a canonical support preserving isomorphism~$\Hh(G,\l)\simeq\Hh(G,\l')$ as in~\cite[Proposition~7.1]{St08}, this follows essentially by transitivity of induction and our results on~$\b$-extensions.  Exactly as in \emph{op. cit.} Proposition 7.2, we have a support preserving isomorphism of algebras~$\Hh(J(\b,\La'),\l')\simeq \Hh(P(\La'_E),\rho)$.  The composition of these isomorphisms with the natural injection~$\Hh(J(\b,\La'),\l')\hookrightarrow \Hh(G,\l')$, gives us an injective map
\[
\Hh(P(\La'_E),\rho)\hookrightarrow \Hh(G,\l),
\]
which preserves support; if~$\phi\in \Hh(P(\La'_E),\rho)$ has support~$P^\circ(\La_E)yP^{\circ}(\La_E)$ for~$y\in P(\La'_E)$ then the corresponding~$\phi_G\in  \Hh(G,\l)$ has support~$J^{\circ}_PyJ^{\circ}_P$.

\subsubsection{Skew covers}
Let~$\pi$ be an irreducible cuspidal representation of~$G$, and consider the set of all such pairs~$([\La,n,0,\b],\th)$ such that~$[\La,n,0,\b]$ is a skew semisimple strata,~$\th\in \Cc_{-}(\La,0,\b)$ and~$\pi$ contains~$\th$.  Choose a pair in this set whose parahoric subgroup~$P^{\circ}(\La_E)$ is minimal under containment relative to all other pairs in the set.  Since there is a unique irreducible representation~$\eta$ of~$J^1(\b,\La)$ containing~$\th$,~$\pi$ must also contain~$\eta$.  Hence, by Lemma~\ref{simpcrit3},~$\pi$ contains a representation~$\l=\k^{\circ}\otimes \rho$ of~$J^{\circ}(\b,\La)$ where~$\k^{\circ}$ is a standard~$\b$-extension of~$\eta$ and~$\rho$ is an irreducible representation of~$J^{\circ}(\b,\La)/J^1(\b,\La)$.  As~$P^{\circ}(\La_E)$ is minimal, it follows that~$\rho$ is cuspidal (\cf~\cite[Lemma~7.4]{St08}). 

Suppose that either~$P^{\circ}(\La_E)$ is not a maximal parahoric subgroup in~$\GE$ or~$\GE$ does not have compact centre.  

\begin{thm}[{\cite[Proposition~7.13]{St08} (\cf~\cite[Appendix~A]{MiSt})}]\label{skewcovers}
There exists an exactly subordinate self-dual decomposition~$V=\bigoplus_{j=-m}^mW^{(j)}$ to~$[\La,n,0,\b]$ such that the pair~$(J^{\circ}_P, \l_P^{\circ})$ is a~$G$-cover of~$(J^{\circ}_P\cap M,\l_P\mid_M)$, where~$J^x{\circ}_P$ is as constructed in Lemma~\ref{lemma72} and~$\k_P$ as in Lemma~\ref{lemma75}.
\end{thm}

The construction follows \emph{mutatis mutandis} that of \emph{op. cit.}, noting that:

\begin{enumerate}
\item We use the results for~$\b$-extensions in positive characteristic from Section~\ref{betasects}, and use Lemma~\ref{simpcrit3} (the characteristic zero version of which is obvious).
\item In the construction of \emph{op. cit.} for a parahoric subgroup~$P^{\circ}(\MM)$ containing~$P^{\circ}(\La_E)$, the proof requires knowledge of the structure of~$\Hh(P(\MM),\rho^{\circ})$ (\cf~Section~$7.2.2$ of \emph{op. cit.}) given by the results of~\cite{Morris93}.  Here we must appeal to Geck--Hiss--Malle's generalisation of the description of the structure of~$\Hh(P(\MM),\rho^{\circ})$ to positive characteristic (see Lemma~\ref{finiteheckedescription}).
\item The proof of \emph{op. cit.} requires the construction of covers
  in general linear groups, namely it
  uses~\cite[Proposition~6.7]{SeSt3}.  For general linear groups, the
  analogous proposition holds in positive characteristic
  (see~\cite[Remarque~2.25]{VA2}).
\item In the definition of \emph{lies over} (cf. \cite[Definition 7.6]{St08}), the use of the word component should be replaced with quotient.
\end{enumerate}


\section{Self-dual and pro-$p$ covers}\label{selfdualcovers}
This section generalises the construction of covers we have give for skew strata to semisimple strata, following~\cite{MiSt}.  Also, inspired by~\cite[Lemme~$5.19$]{VA3}, we define pro-$p$ covers at the level of the~$J^1$ groups.  These results will not be used in the rest of the paper, and are included with future work in mind.

Let~$M$ be a Levi subgroup of~$G$ which is the stabiliser of the self-dual decomposition~$V=\bigoplus_{j=-m}^mW^{(j)}$.  Letting~$\tG^{(j)}=\Aut_F(W^{(j)})$ and~$G^{(0)}=\Aut_F(W^{(0)})\cap G$ we have~$M=G^{(0)}\times \prod_{j=1}^m \tG^{(j)}$.  Let~$\t=\t^{(0)}\otimes\bigotimes_{j=1}^m\ttau^{(j)}$ be a cuspidal irreducible representation of~$M$.  Let~$\Mm$ denote the stabiliser of~$V=\bigoplus_{j=-m}^mW^{(j)}$ in~$A$.

\begin{lemma}[{\cite[Proposition~8.10]{Finitude},~\cite[Proposition 5.1]{MiSt}}]
There are a self-dual semisimple stratum~$[\La,n,0,\b]$ with~$\b\in \Mm$ and a self-dual semisimple character~$\th$ of~$H^1(\b,\La)$ such that~$V=\bigoplus_{j=-m}^mW^{(j)}$ is properly subordinate to~$[\La,n,0,\b]$ and 
\[
\th\mid_{H^1(\b,\La)\cap M}=\th^{(0)}\otimes\bigotimes_{j=1}^m \(\tth^{(j)}\)^2,
\]
with~$\th^{(0)}$ contained in~$\t^{(0)}$ and, for each~$j>0$,~$\(\tth^{(j)}\)^2$ contained in~$\ttau_j$ where we have identified~$H^1(\b,\La)\cap M$ with~$H^1(\b^{(0)},\La^{(0)})\times\prod_{j=1}^m\tH^1(\b^{(j)},\La^{(j)})$ as in Lemma~\ref{lemma75}.
\end{lemma}

Let~$\rho$ be an irreducible cuspidal representation of~$M^{\circ}(\La_E)=J^{\circ}_P/J^1_P\simeq J^\circ_L/J^1_L$.  We can form the representations~$\l_P^{\circ}=\k_P\otimes\rho$ of~$J_P^{\circ}$ and~$\l_L^{\circ}=\k_L\otimes\rho$ of~$J_L^{\circ}$ by inflation.   

\begin{thm}[{\cite[Theorem~$4.3$]{MiSt} (\cf~also~\cite[Proposition~$7.13$]{St08})}]\label{Gcovers}
The pair~$(J_P^\circ,\l_P^\circ)$ is a~$G$-cover of~$(J_L^\circ,\l_L^\circ)$ relative to~$P$.
\end{thm}

The proof generalises to positive characteristic with the same adaptions as commented on the proof of Theorem~\ref{skewcovers}.

\begin{thm}\label{Gcovers2}
The pair~$(J^1_P,\eta_P)$ is a~$G$-cover of~$(J^1_L,\eta_L)$ relative to~$P$.
\end{thm}

\begin{proof}
By~\cite[Page~246,~(0.5)]{Blondel}, it is equivalent to show that; for all smooth~$R$-representations~$\pi$ of~$G$ the map of vector spaces
\[
\Phi: \R_{\k_P}(\pi)\to \R_{\k_L}(r^G_P(\pi)),
\]
given by~$\Phi(f)=r^G_P\circ f$ for~$f\in \R_{\k_P}(\pi)$, is injective.  This map is easily checked to be a homomorphism of representations of~$M^{\circ}(\La_E)$. Assume~$\ker(\Phi)$, the kernel of~$\Phi$, is non-zero and let~$\phi$ be an irreducible subrepresentation of~$\ker(\Phi)$.  Let~$(\ov{\tau},\ov{L})$ be in the cuspidal support of~$\phi$, here we mean that~$\phi$ is a quotient of~$i^G_P(\ov{\tau})$.

Thus~$\ov{L}$ is a Levi subgroup of~$M^{\circ}(\La_E)$ (we allow the case~$\ov{L}=M^{\circ}(\La_E)$).  Let~$\ov{P}$ be the standard parabolic subgroup of~$M^{\circ}(\La_E)$ containing~$\ov{L}$ with Levi decomposition~$\ov{P}=\ov{L}\,\ov{U}$.  Choose a self-dual~$\o_E$-lattice sequence~$\La'$ such that~$P^{\circ}(\La'_E)$ is equal to the preimage of~$\ov{P}$ under the projection~$P^{\circ}(\La_E)\to M^{\circ}_P$ and such that~$P^{\circ}(\La)\supseteq P^{\circ}(\La')$ (considering~$\La$ and~$\La'$ as~$\o_F$-lattice sequences), this is possible by~\cite[Lemma~2.8]{St08}.  Let~$\k'=b_{\La,\La'}(\k)$.  The decomposition of~$V=\bigoplus_{j=-m}^mW_j$ is exactly subordinate to the~$[\La',n',0,\b]$.  Hence we can form the groups 
\[
J'_P=H^1(\b,\La')(J^{\circ}(\b,\La')\cap P),\quad J'_L=J'_P\cap L
\]
and the representations~$\k'_P$ of~$J'_P$ (the natural representation on the~$(J^{\circ}(\b,\La')\cap U)$-fixed vectors of~$\k'$) and~$\k'_L=\k'_P\mid_{J'_ L}$.   

We have the left exact sequence
\[
0\to \om\to\R_{\k_P}(\pi)\to \R_{\k_L}(r^G_P(\pi)).
\]
We apply the Jacquet functor~$r^{M(\La_E)}_{\ov{P}}$ (which is exact) and have
\[
0\to r^{M(\La_E)}_{\ov{P}}(\om)\to\R_{\k'_P}(\pi)\to \R_{\k'_L}(r^G_P(\pi)),
\]
as~$r^{M(\La_E)}_{\ov{P}}\circ \R_{\k_P}(\pi)   \simeq \R_{\k'_P}(\pi)$ and~$r^{M(\La_E)}_{\ov{P}}\circ\R_{\k_L}(r^G_P(\pi))\simeq \R_{\k'_L}(r^G_P(\pi))$ by compatibility of~$\k$ and~$\k'$. Then, taking the~$\ov{\tau}$-isotypic components (which is a left exact functor) we have an exact sequence
\[
0\to \Hom_{\ov{L}}(\ov{\tau},r^{M(\La_E)}_{\ov{P}}(\om))\to\Hom_{\ov{L}}(\ov{\tau},\R_{\k'_P}(\pi))\to \Hom_{\ov{L}}(\ov{\tau},\R_{\k'_L}(r^G_P(\pi))).
\]
By right adjointness of~$\R_{\k'_P}$ and~$\R_{\k'_L}$ with~$\I_{\k'_P}$ and~$\I_{\k'_L}$ and right adjointness of restriction with compact induction this is isomorphic to the exact sequence
\[
0\to  \Hom_{\ov{L}}(\ov{\tau},r^{M(\La_E)}_{\ov{P}}(\om))\to \Hom_{J^{\circ}_P}(\k'_P\otimes\ov{\tau},\pi)\to \Hom_{J^{\circ}_L}(\k'_L\otimes \ov{\tau},r^G_P(\pi))
\]
As~$\om$ contains a subrepresentation with cuspidal support~$\ov{\tau}$,~$\Hom_{\ov{L}}(\ov{\tau},r^{M(\La_E)}_{\ov{P}}(\om))\neq 0$. However, by Theorem~\ref{Gcovers},~$(J'_P,\k'_P\otimes\ov{\tau})$ is a~$G$-cover of~$(J'_L,\k'_L\otimes\ov{\tau})$ relative to~$P$. Hence, by~\cite[Page~246,~(0.5)]{Blondel}, the map~$\Hom_{J^{\circ}_P}(\k'_P\otimes\ov{\tau},\pi)\to \Hom_{J^{\circ}_L}(\k'_L\otimes \ov{\tau},r^G_P(\pi))$ is injective, a contradiction.  
\end{proof}
%
%

\section{Quasi-projectivity of types}\label{Qproj}
This section shows that the types we consider are quasi-projective, so that Theorem \ref{quasiprojvigneras} applies.

\begin{lemma}\label{quasprojlemma} Suppose that~$n$ is a distinguished double coset representative of~$P_{\Upsilon}^\circ\backslash \GE/P^{\circ}_{\Upsilon}$ with projection~$w$ in the affine Weyl group of~$\GE$ such that, if~$P^{\circ}_{\Upsilon}$ corresponds to the subset~$K$ of the fundamental reflections in the affine Weyl group~$W'$ (cf. Section \ref{Levzerosection}), then~$wK=K$. Let~$\tau$ be a representation of~$M^{\circ}(\Upsilon_E)$.  Then, we have an isomorphism of vector spaces
\[\Hom_{J^1_{\Y}}(\kappa^{\circ}_{\Y},\ind^{J^{\circ}_{\Upsilon}}_{J^{\circ}_{\Upsilon}\cap (J^{\circ}_{\Upsilon})^n}(\kappa^{\circ}_{\Upsilon}\otimes\tau)^n)\simeq 
\tau^n,\]
which is an isomorphism of representations if~$n\in I_G(\kappa_\Y^{\circ})$.
\end{lemma}

\begin{proof}
Observe that we have~$J^{\circ}_{\Y}=J_{\Y}^1(J^{\circ}_{\Y}\cap (J^{\circ}_{\Y})^n)\supseteq J_{\Y}^1(P^{\circ}_{\Y}\cap (P^{\circ}_{\Y})^n)$ and moreover~$J^{\circ}_{\Y}/J^1_{\Y}=J_{\Y}^1(P^{\circ}_{\Y}\cap (P^{\circ}_{\Y})^n)/J^1_{\Y}$, as~$wK=K$ (and using Section \ref{Levzerosection} \ref{Part1levzero}).  Therefore
 \[J^{\circ}_{\Y}= J_{\Y}^1(P^{\circ}_{\Y}\cap (P^{\circ}_{\Y})^n).\]
Thus, by Mackey theory, we have
\[\Res_{J^1_{\Y}}^{J^{\circ}_{\Y}}(\ind^{J^{\circ}_{\Y}}_{J^{\circ}_{\Y}\cap (J^{\circ}_{\Y})^n}(\kappa^{\circ}_{\Y}\otimes\tau)^n)\simeq\ind^{J^1_{\Y}}_{J^1_{\Y}\cap (J^{\circ}_{\Y})^n}(\kappa^{\circ}_{\Y}\otimes\tau)^n.\]
Therefore, we have isomorphisms of vector spaces
\begin{align*}
\Hom_{J^1_{\Y}}(\kappa^{\circ}_{\Y},\ind^{J^{\circ}_{\Upsilon}}_{J^{\circ}_{\Upsilon}\cap (J^{\circ}_{\Upsilon})^n}(\kappa^{\circ}_{\Upsilon}\otimes\tau)^n)
&\simeq \Hom_{J^1_{\Y}}(\kappa^{\circ}_{\Y},\ind^{J^1_{\Upsilon}}_{J^1_{\Upsilon}\cap (J^{\circ}_{\Upsilon})^n}(\kappa^{\circ}_{\Upsilon}\otimes\tau)^n)\\
&\simeq \Hom_{J^1_{\Y}\cap (J^{\circ}_{\Y})^n}(\kappa^{\circ}_{\Y},(\kappa^{\circ}_{\Upsilon})^n\otimes\tau^n)
\end{align*}
which, checking actions, is actually an isomorphism of representations of~$M_{\Y}^{\circ}$, where the action of \[M_{\Y}^{\circ}=J^1_{\Y}(J^{\circ}_{\Y}\cap (J^{\circ}_{\Y})^n)/J^1_{\Y}\simeq (J^{\circ}_{\Y}\cap (J^{\circ}_{\Y})^n)/(J^1_{\Y}\cap (J^{\circ}_{\Y})^n)\] on homomorphisms in~$\Hom_{J^1_{\Y}\cap (J^{\circ}_{\Y})^n}(\kappa^{\circ}_{\Y},(\kappa^{\circ}_{\Upsilon})^n\otimes\tau^n)$ is given in the usual way by pre-composition with~$(\kappa^{\circ}_{\Y})^{-1}$ and post-composition with~$(\kappa^{\circ}_{\Upsilon})^n\otimes\tau^n$.  By Lemma \ref{intetakappa} we can choose~$S\in\Hom_{J^1_{\Y}\cap (J^{\circ}_{\Y})^n}(\kappa^{\circ}_{\Y},(\kappa^{\circ}_{\Upsilon})^n)$ nonzero, and~$\Hom_{J^1_{\Y}\cap (J^{\circ}_{\Y})^n}(\kappa^{\circ}_{\Y},(\kappa^{\circ}_{\Upsilon})^n)=\Hom_{J^1_{\Y}\cap (J^1_{\Y})^n}(\kappa^{\circ}_{\Y},(\kappa_{\Upsilon}^{\circ})^n)\simeq R$ by Theorem \ref{intetas}.  Hence, by Lemma~\ref{BK532} (applied with~$X_1=X_1^1=J^1_{\Y}$,~$X_2=J_{\Y}^n$,~$X_2^1=(J_{\Y}^1)^n$,~$\mu_1=\eta_{\Y}$,~$\mu_2=\kappa_{\Y}^n$,~$\zeta_1=1$, and~$\zeta_2=\tau^n$) we have an isomorphism of vector spaces
\[ \Hom_{J^1_{\Y}\cap (J^{\circ}_{\Y})^n}(1,\tau^n)\rightarrow \Hom_{J^1_{\Y}\cap (J^{\circ}_{\Y})^n}(\kappa^{\circ}_{\Y},(\kappa^{\circ}_{\Upsilon})^n\otimes\tau^n), \]
given by the tensor product with~$S$ which is an isomorphism if~$S\in \Hom_{J^{\circ}_{\Y}\cap (J^{\circ}_{\Y})^n}(\kappa^{\circ}_{\Y},(\kappa^{\circ}_{\Upsilon})^n)$, which will be the case if~$\Hom_{J^{\circ}_{\Y}\cap (J^{\circ}_{\Y})^n}(\kappa^{\circ}_{\Y},(\kappa^{\circ}_{\Upsilon})^n)\neq 0$, i.e. if~$n\in I_G(\kappa_\Y^{\circ})$. Moreover, as a representation of~$M^{\circ}_{\Y}=(J^{\circ}_{\Y}\cap (J^{\circ}_{\Y})^n)/(J^1_{\Y}\cap (J^{\circ}_{\Y})^n)$,
\[ \Hom_{J^1_{\Y}\cap (J^{\circ}_{\Y})^n}(1,\tau^n)\simeq  \tau^n.\]
\end{proof}

It seems likely that the elements~$n$ considered in Lemma \ref{quasprojlemma} do intertwine~$\kappa_{\Y}^{\circ}$, we do not prove this here as it is not needed for our application.

\begin{thm}\label{quasprojtheorem}
Suppose~$\tau$ is cuspidal.  The representation~$\I_{\kappa_{\Y}^{\circ}}(\tau)$ is quasi-projective.
\end{thm}

\begin{proof}
Notice that, as~$J^1_{\Upsilon}$ is pro-$p$, the~$\eta$-isotypic component of~$\I_{\kappa^{\circ}_{\Upsilon}}(\tau)$ is a summand of the restriction of~$\I_{\kappa^{\circ}_{\Upsilon}}(\tau)$ to~$J_{\Upsilon}$, and no representation in its complement contains~$\eta$, whence cannot be isomorphic to~$\lambda=\kappa^{\circ}_{\Upsilon}\otimes\tau$.  However, we have~$\I_{\kappa^{\circ}_{\Upsilon}}(\tau)^\eta\simeq\kappa^{\circ}_{\Upsilon}\otimes \R_{\kappa^{\circ}_{\Upsilon}}\circ\I_{\kappa^{\circ}_{\Upsilon}}(\tau)$ (\cf~\cite[Lemme~2.6]{VA2}).  We can decompose~$\R_{\k^{\circ}_{\Y}}\circ\I_{\k^{\circ}_{\Y}}(\t)$ as a direct sum and choose distinguished double cosets for each summand as in the proof of Theorem \ref{maincorollary}.  By Lemmas \ref{lemmaifzerothen} and \ref{quasprojlemma}, the summands are either zero (when the distinguished coset representative projects to an element~$w$ with~$wK\neq K$), or have the same dimension of~$\tau$.  Hence the~$\kappa^{\circ}_{\Upsilon}\otimes\tau$-isotypic component must be a direct summand of the~$\eta$-isotypic component of~$\I_{\kappa^{\circ}_{\Upsilon}}(\tau)$ and, by Lemma \ref{quasiprojlemma}, the representation~$\I_{\kappa^{\circ}_{\Upsilon}}(\tau)$ is quasi-projective.
\end{proof}

\section{Exhaustion}\label{Exhaust}

We show how Corollary~\ref{maincorollary} can be used to show certain
representations of~$G$ we have constructed are irreducible and
cuspidal. Moreover, with Theorem~\ref{skewcovers}, we show that this
construction exhausts all irreducible cuspidal representations
of~$G$. In the complex case this construction is the same
as~\cite[Corollary~6.19]{St08}.  However, in addition to extending
this construction to~$\ell$-modular representations,
Corollary~\ref{maincorollary} allows us to make some comparisons between
certain irreducible cuspidal representations in our exhaustive lists.

We call a skew semisimple stratum~$[\La,n,0,\b]$ \emph{cuspidal} if~$\GE$ has compact centre and~$P^{\circ}(\La_E)$ is a maximal parahoric subgroup.  A \emph{type} for~$G$ is a pair~$(J,\k\otimes\t)$ where~$J=J(\b,\La)$ for some self-dual semisimple stratum~$[\La,n,0,\b]$,~$\k$ is a~$\b$-extension of the unique Heisenberg representation~$\eta$ containing~$\th\in\Cc_{-}(\La,0,\b)$ and~$\t$ is an irreducible representation of~$J/J^1$ with cuspidal restriction to $J^{\circ}/J^1$.  We call a type~$(J,\k\otimes\t)$ \emph{cuspidal} if~$[\La,n,0,\b]$ is a cuspidal stratum.  We call a cuspidal type~$(J,\k\otimes\t)$ \emph{supercuspidal} if~$\t$ is supercuspidal on restriction to~$J^{\circ}/J^1$.

\begin{thm}\label{cusptype}
Let~$(J,\k\otimes\t)$ be a cuspidal type for~$G$ relative to the skew semisimple stratum~$[\La,n,0,\b]$, then~$\I_{\k}(\t)$ is irreducible and cuspidal.
\end{thm}

\begin{proof}
The conditions on~$[\La,n,0,\b]$ guarantee that~$P(\La_E)$ is its own normaliser.  By Corollary~\ref{maincorollary},~$\End_G(\I_{\k}(\t))\simeq R$. Let~$\pi$ be an irreducible~$R$-representation of~$G$ such that~$\k\otimes\t$ is a subrepresentation of~$\pi$ (hence~$\pi$ is a quotient of~$\I_{\k}(\t)$).  We must show that~$\k\otimes\t$ is also a quotient of~$\pi$ in order to apply Lemma~\ref{simpcrit1}.  As~$J^1$ is pro-$p$, we can decompose~$\pi\simeq \pi^\eta\oplus \pi(\eta)$ where~$\pi^\eta$ denotes the~$\eta$-isotypic component of~$\pi$ and no subquotient of~$\pi(\eta)$ contains~$\eta$.   By Corollary~\ref{maincorollary}, we have~$\I_{\k}(\t)^\eta\simeq \k\otimes \t$, and hence by exactness~$\pi^\eta\simeq  \k\otimes \t$ (or zero which it can't be as~$\k\otimes\t$ is a subrepresentation of~$\pi$).   Therefore, by Lemma~\ref{simpcrit1},~$\I_{\k}(\t)$ is irreducible. Cuspidality follows from a classical argument (\cf~\cite[\S1]{Carayol} and~\cite[\S2,~2.7]{Vig96}).
\end{proof}

\begin{thm}\label{exhaustion}
Every irreducible cuspidal representation of~$G$ contains a cuspidal type.
\end{thm}

\begin{proof}
Let~$\pi$ be an irreducible cuspidal representation of~$G$.  By~\cite[Theorem~5.1]{St05}, the proof of which applies in positive characteristic~$\ell\neq p$, there exist a skew semisimple stratum~$[\La,n,0,\b]$ and~$\th\in\Cc_{-}(\La,0,\b)$ such that~$\pi$ contains~$\th$.  Thus~$\pi$ contains the unique extension~$\eta$ of~$\th$ to~$J^1$.  Let~$\k$ be a standard~$\b$-extension of~$\eta$.  By Lemma~\ref{simpcrit3}, the functor~$\k\otimes -$ identifies the category of~$R$-representations of~$M(\La_E)$ with the category of~$\eta$-isotypic representations of~$J$.  Thus~$\pi$ contains~$\k\otimes \t$ for some irreducible representation~$\t$ of~$J/J^1$. The proof now follows, using~\cite[II~10.1]{vignerasselecta}, from Theorem~\ref{skewcovers} (\cf~\cite[Appendix~A]{MiSt} and~\cite[Theorem~7.14]{St08}). 
\end{proof}

A consequence of the of Corollary~\ref{maincorollary} is the following intertwining implies conjugacy theorem:

\begin{thm}\label{cuspcomp}
Suppose~$(J_{\La},\k_{\La}\otimes\t_{\La})$ and~$(J_{\Y},\k_{\Y}\otimes \t_{\Y})$ are cuspidal types associated to the semisimple strata~$[\La,n_{\La},0,\b]$ and~$[\Y,n_{\Y},0,\b]$.  If~$I_{\k_{\La}}(\t_{\La})\simeq \I_{\k_{\Y}}(\t_{\Y})$, then there exists~$g\in G_E$ such that~$(J_{\Y}^g,\k_{\Y}^g\otimes \t_{\Y}^g)=(J_{\La},\k_{\La}\otimes\t_{\La})$.
\end{thm}

\begin{proof}
By Corollary~\ref{maincorollary}~\ref{maincorollarypart1} and Corollary \ref{levelzerocorollaries}~\ref{part2levelzerocorollaries}, the lattice sequences~$\La_E$ and~$\Y_E$ are in the same~$\GE$-orbit.  Hence, by conjugating by an element of~$\GE$ if necessary, we can assume~$\La=\Y$.  Hence the groups of the cuspidal types coincide, and by twisting~$\t_{\La}$ by a character~$\chi$ of~$M(\La_E)$ if necessary, we can assume~$\k_{\La}=\k_{\Y}$.  By Corollary~\ref{maincorollary}~\ref{maincorollarypart2} and adjointness, we have
\[\Hom_{M(\Y_E)}(I_{\k_{\Y}}(\chi\otimes\t_{\La}), \I_{\k_{\Y}}(\t_{\Y}))=\Hom_{M(\Y_E)}(\chi\otimes\t_{\La},\t_{\Y}),\]
which is non-zero by hypothesis.  Hence~$\chi\otimes\t_{\La}\simeq \t_{\Y}$ and thus the cuspidal types are conjugate by an element of~$\GE$.
\end{proof}%
\bibliographystyle{plain}
\bibliography{CuspClass}

\end{document}